\documentclass[11pt]{article}
\usepackage[latin1]{inputenc}
\usepackage[T1]{fontenc}
\usepackage{amsmath}
\usepackage{amsthm}
\usepackage{amssymb}
\usepackage{mathrsfs}
\usepackage{dsfont}
\usepackage{stmaryrd}
\usepackage{amscd}
\usepackage{float}
\usepackage[a4paper]{geometry}
\usepackage{graphicx}
\usepackage{subfigure}
\usepackage{epsfig}
\usepackage{lmodern}
\usepackage{pstricks}
\usepackage[colorlinks=true]{hyperref} 
\usepackage{verbatim}
 \usepackage[colorinlistoftodos,french]{todonotes}
\title{\bf The missing $\boldsymbol{(A,D,r)}$ diagram}
\author{Alexandre Delyon\footnote{Universit\'e de Lorraine, CNRS, Institut Elie Cartan de Lorraine, BP 70239 54506 Vand\oe uvre-l\`es-Nancy Cedex, 
France ({\tt alexandre.delyon@univ-lorraine.fr}).}
	\and Antoine Henrot\footnote{Universit\'e de Lorraine, CNRS, Institut Elie Cartan de Lorraine, BP 70239 54506 Vand\oe uvre-l\`es-Nancy Cedex, France ({\tt antoine.henrot@univ-lorraine.fr}).}
	\and Yannick Privat\footnote{IRMA, Universit\'e de Strasbourg, CNRS UMR 7501, Inria, 7 rue Ren\'e Descartes, 67084 Strasbourg, France ({\tt yannick.privat@unistra.fr}).}
}

\newcommand{\R}{\mathbb{R}}
\newcommand{\N}{\mathbb{N}}

\newcommand{\eps}{\varepsilon}

\newcommand{\h}{\mathcal{H}^{n-1}}

\makeatletter
\DeclareFontFamily{U}{tipa}{}
\DeclareFontShape{U}{tipa}{m}{n}{<->tipa10}{}
\newcommand{\arc@char}{{\usefont{U}{tipa}{m}{n}\symbol{62}}}%

\newcommand{\arc}[1]{\mathpalette\arc@arc{#1}}

\newcommand{\arc@arc}[2]{%
  \sbox0{$\m@th#1#2$}%
  \vbox{
    \hbox{\resizebox{\wd0}{\height}{\arc@char}}
    \nointerlineskip
    \box0
  }%
}
\makeatother

\renewcommand{\geq}{\geqslant}
\renewcommand{\leq}{\leqslant}

\newtheorem{theorem}{Theorem}

\newtheorem{proposition}{Proposition}

\newtheorem{definition}{Definition}
\newtheorem{lemma}{Lemma}
\theoremstyle{definition}
\theoremstyle{definition}\newtheorem{remark}{Remark}

\begin{document}
\maketitle

\begin{abstract}
In this paper we are interested in "optimal" universal geometric inequalities involving the area, diameter and inradius of convex bodies. The term 
"optimal" is to be understood in the following sense: we tackle the issue 
of minimizing/maximizing the Lebesgue measure of a convex body among all convex sets of given diameter and inradius.  {The minimization problem in the two-dimensional case has been solved in a previous work, by M.~Hernandez-Cifre and G.~Salinas. In this article, we provide a generalization to the $n$-dimensional case based on a different approach, as well as the complete solving of the maximization problem in the two-dimensional case.} This allows us to completely determine the so-called 2-dimensional Blaschke-Santal\'o diagram 
for planar convex bodies with respect to the three magnitudes area, diameter and inradius in euclidean spaces, denoted $(A,D,r)$. Such a diagram is used to determine the  range of possible values of the area of convex sets depending on their diameter and inradius. Although this question of convex geometry appears quite elementary, it had not been answered until now. This is likely related to the fact that the diagram description uses unexpected particular convex sets, such as a kind of smoothed nonagon inscribed in an equilateral triangle.
\end{abstract}


\medskip


%\tableofcontents
%\listoftodos
\section{Introduction}

%\todo[inline, color=red!30]{Revoir l'intro intro avec état de l'art}
Let $n\in \N^*$. In the whole article, we will denote by $\mathcal K_n$ the set of all convex bodies (i.e. compact convex sets with non-empty interior) in $\R^n$. 

In convex geometry, the search for optimal inequalities between the six
standard geometrical quantities which are the surface $A$ (or volume $V$), the perimeter $P$, the diameter $D$, the inradius $r$, the circumradius 
$R$ and the (minimal) width\footnote{In other words, the smallest distance between any two different parallel supporting hyperplanes of a convex body.} $w$ of any convex body,
is a very old activity that dates back to the work of W. Blaschke (\cite{Blaschke1},
\cite{Blaschke2}) and has been extensively studied by L. Santal\'o in \cite{Santalo}.
For a list of such inequalities known in 2000, we refer to the classical review paper \cite{awyong}.The general idea is to consider three of the aforementioned quantities $(q_1,q_2,q_3)$ and to determine a complete system of
inequalities relating them, in other words a system of inequalities describing the set
$$
\{(q_1(K),q_2(K),q_3(K)), \ K\in \mathcal K_n\}.
$$ 
In general, it is convenient to summarize it into a diagram, usually called {\it Blaschke-Santal\'o diagram}. It represents the set of possible values of the triple that can be reached by a convex set (suitably normalized).
Among the 20 possible choices of this three geometric quantities, L. Santal\'o completely solved in his work the 6 cases $(A,P,w)$, $(A,P,r)$, $(A,P,R)$, $(A,D,w)$, $(P,D,w)$, $(D,r,R)$ and gave a partial solution to $(D,R,w)$ and $(r,R,w)$. These two last cases were eventually solved by M. Hernandez Cifre and S. Segura Gomis in \cite{Hernandez00DCG}. In a series 
of papers with collaborators, M. Hernandez Cifre
has also been able to prove complete systems of inequalities in the cases 

$(A,D,R)$, $(P,D,R)$ \cite{Hernandez02Arch}, in the cases 
$(A,r,R)$, $(P,r,R)$ \cite{Hernandez03Mon} and finally
in the case $(D,r,w)$ \cite{Hernandez00AMM}.

In spite of all these efforts, several Blaschke-Santal\'o diagrams (or
complete systems of inequalities) remain unknown.
To the best of our knowledge, this is the case for the diagrams
$(A,P,D)$, $(A,D,r)$, $(A,r,w)$, $(A,R,w)$, $(P,D,r),$
$(P,r,w)$ and $(P,R,w)$. Let us mention that several
 interesting inequalities for $(P,D,r)$ and $(P,R,w)$ can be found in \cite{Hernandez01}. {Let us also mention several works dedicated to Blaschke-Santal\'o diagrams involving four geometric quantities (see e.g. \cite{Brandenberg}).}

In this paper, we focus on the case $(A,D,r)$ and completely solve it in the two-dimensional case ($n=2$), and partially in the general case $n\geq 3$.
More precisely in the case $n=2$, we obtain universal inequalities involving the area of a plane convex set, its diameter and inradius, and we plot the corresponding Blaschke-Santal\'o diagram:
$$
\mathcal{D}=\left\{(x,y)\in \mathbb{R}^2, \ x=2\frac{r(K)}{D(K)}, \ y=\pi \frac{r^2(K)}{A(K)}, \ K\in \mathcal K_2\right\}.
$$
To this aim, we will introduce two families of optimization problems for the area
(or the volume in higher dimension) and then solve them. More precisely, we will tackle the issue of maximizing and minimizing the area
with prescribed diameter and inradius. It turns out that the minimization 
problem has already been solved in the two dimensional case by M. Hernandez Cifre and G. Salinas \cite{Hernandez01}. The optimal set is known to be a two-cap body defined as the convex hull of a disk of radius $r$ {and illustrated on Fig.~\ref{fig:2cap}}.
with two points that are symmetric with respect to the center of the ball 
and at a distance $D$. 
This result has been extended in three dimensions in \cite{YangZhang} but 
with an
additional assumption. In this paper, we solve this minimization problem in full
generality (see Theorem \ref{theo:SOmin}).

\begin{figure}[h!]
\begin{center}
\includegraphics[width=6cm]{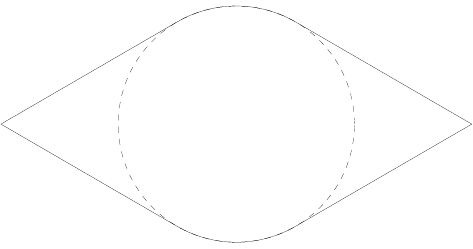}
\caption{{The two-cap body in 2D, minimizer of the area among convex bodies of prescribed inradius and diameter.} \label{fig:2cap}}
\end{center}
\end{figure}

Regarding the maximization problem, it is much harder and we are  only able to solve it in the two-dimensional case. At first glance, it seems intuitive that the optimal shape should be a {\it spherical slice} defined as the intersection of a disk of diameter $D$ with a strip of width $2r$, symmetric with respect to the center of the disk {(see Fig. \ref{Def2})}. Surprisingly, this is only true for "large" values of {$D/r$} (more precisely for $D\geq \alpha r$ with $\alpha \simeq 2.388$, see Theorem \ref{theo:SOmax}), {while for small values 
of $D/r$} the optimal set is some kind of nonagon made of 3 segments and 6 arcs of circle inscribed in an equilateral triangle {(see Fig.~\ref{figTD2})}. For the precise definition of this set, we refer to Definition \ref{defKE} hereafter. It is likely that this unexpected solution explains why this elementary {\it shape optimization}
problem remained unsolved up to now.

The article is organized as follows. Section \ref{sec:mainRes} is devoted 
to introducing the optimization problems we will deal with and stating the main results. In Section \ref{sectionBS}, the Blaschke-Santal\'o diagram
$\mathcal{D}$ for the triple $(A,D,r)$ is plotted. The whole sections \ref{sec:proofTheoMin} and \ref{sec:proofTheoMax} are respectively concerned 
with the proofs of Theorems~\ref{theo:SOmin} and  \ref{theo:SOmax}. Because of the variety and complexity of optimizers, the proofs appear really difficult and involve several tools of convex analysis, optimal control and geometry.  

\medskip

Let us end this section by gathering some notations used throughout this article:
\begin{itemize}
\item $\mathcal{H}^{n-1}$ is the $n-1$ dimensional Hausdorff measure.
\item if $K$ is a convex set of $\R^2$, we call respectively $A(K)$, $D(K)$ and $r(K)$ (or alternatively $A$, $D$ and $r$ if there is no ambiguity) the area, diameter and inradius of $K$.
\item in the more general $n$-dimensional case, we keep the same notations, except for the volume of $K$ which will be either denoted $V(K)$ or $|K|$.
\item $x\cdot y$ is the Euclidean inner product of two vectors $x$ and $y$ in $\R^n$.
\item $B(O,r)$ denotes the ball of center $O$ and radius $r$ while $S(O,r)$ is the sphere
(its boundary). 
\item The boundary of the biggest ball included into a convex set will be 
called
 {\it incircle} in dimension 2, {\it insphere} in higher dimension.
\end{itemize}
\subsection{Optimization problems and main results}\label{sec:mainRes}
Let us first make the notations precise.
Let $r>0$, $D>2r$ be given and let $\mathcal{K}^{n}_{r,D}$ be the set of convex bodies of $\R^{n}$ having as inradius $r$ and as diameter $D$, namely 
$$
\mathcal{K}^{n}_{r,D}=\{K\in \mathcal K_n \mid r(K)=r\text{ and } D(K)=D\}.
$$
We are interested in the following maximization problem
\begin{equation}\label{SOP}\tag{$\mathcal{P}_{\rm max}$}
\boxed{\sup_{K\in \mathcal{K}^{2}_{r,D}}\vert K \vert}
\end{equation}
and minimization problem
\begin{equation}\label{SOP2}\tag{$\mathcal{P}_{\rm min}$}
\boxed{\inf_{K\in \mathcal{K}^{n}_{r,D}}\vert K \vert.}
\end{equation}

Note that the condition $D>2r$ guarantees that the set $\mathcal{K}^n_{r,D}$ is non-empty. If $D=2r$, problems are obvious since only the ball belongs to the set of constraints $\mathcal{K}^{n}_{r,D}$. 

Let us first observe, since we are working with convex sets, that existence of solutions for Problems \eqref{SOP} and \eqref{SOP2} is almost straightforward.
\begin{proposition}\label{prop:exist}
Let $(r,D)$ be two given parameters such that $D> 2r$.
Problems \eqref{SOP} and \eqref{SOP2} have a solution.
\end{proposition}
\begin{proof}
{Without loss
of generality, by using an easy rescaling argument, one can deal with sets of constraints with unitary inradius, in other words $r=1$ and with diameter $D> 2$.}

Let us deal with the minimization problem \eqref{SOP2}, the case of the maximization problem \eqref{SOP} being exactly similar. Let us consider a minimizing sequence $(K_{m})_{m\in \N}$. Since we are working with sets of diameter $D$, up to applying a well-chosen translation to each element of the sequence, on can assume that every convex set $K_{m}$ is included in a (compact) box $B$ of $\R^n$. Since the set of convex sets included in a given box is known to be compact for the Hausdorff distance \cite{HenrPierre}, there exists a subsequence (still denoted $(K_{m})_{m\in \N}$) converging to a convex set $K$. {To conclude, we will prove that the 
objective function (the area) is continuous with respect to the Hausdorff 
distance and that the diameter and inradius constraints are stable for the Hausdorff convergence, in other words that $K$ belongs to the admissible set $\mathcal{K}^{2}_{r,D}$. Recall that the volume and diameter functionals are not continuous in general for the Hausdorff distance. Nevertheless, when dealing with convex sets, the continuity property becomes true (see \cite{HenrPierre,Schneider}). }

{It remains to show that the inradius constraint is also continuous for the Hausdorff distance. Let $(K_{m})_{m\in \N}$ be a sequence of convex bodies converging to $K$ for the Hausdorff distance. Let us introduce $r_m=r(K_m)$, $r=r(K)$ and $x_m\in K_m$, such that $B(x_m,r_m)\subset 
K_m$. Since $(r_m)$ (resp. $(x_m)$) is bounded, there exists subsequences 
still denoted $r_m$  and $x_m$ with a slight abuse of notation, that converges respectively towards $\tilde{r}\geq 0$ and $\tilde x\in \R^{n}$. By 
stability of the Hausdorff convergence for the inclusion (see e.g. \cite[Chapter~2 and Prop.~2.2.17]{HenrPierre}), we have $B(\tilde x, \tilde r)\subset K$. Therefore, one has $\tilde r\leq  r$.  Assume by contradiction 
that $r>\tilde r$. Hence, there exists  $x\in K$ and $\alpha>0$ such that 
$B(x,\tilde r+\alpha)\subset K$. 
Let us consider the closed disk $\hat B=B(x,(\tilde r+\alpha)/2)$.
By stability of the Hausdorff convergence, one has $B\subset K_{m}$ whenever $m$ is large enough, which implies that $r(K_m)\geq (1+\alpha)/2$, yielding to a contradiction. The expected continuity property follows.
} 

%Assume, without loss of generality, that $r(K_{m})=1$ for every $m\in \N$. By applying  well-chosen translations of $K_m$, one can moreover assume that $B(0,1)\subset K_{m}$ for every $m\in \N$. By stability of the inclusion for the Hausdorff convergence, one gets $B(0,1)\subset K$ and $r(K)\geq 1$. Assume by contradiction that $r(K)>1$. Hence, there exists  $x\in K$ and $\alpha>1$ such that $B(x,\alpha)\subset K$. 
%Let us consider the closed disk $\hat B=B(x,(1+\alpha)/2$.
%By stability of the Hausdorff convergence \cite{HenrPierre}, one has $B\subset K_{m}$ whenever $m$ is large enough, which implies that $r(K_n)\geq (1+\alpha)/2$. We have then reached a contradiction. 
\end{proof}

As underlined in the Introduction, Problem \eqref{SOP2} has already been solved in the two-dimensional case in \cite{Hernandez01}. In what follows, we will generalize it to the general case $\R^{n}$, by  proving that the two-cap body is the only solution in any dimension.
\begin{theorem}\label{theo:SOmin}
{The (unique) optimal shape for Problem \eqref{SOP2} is the convex hull of a ball of radius $r$ and two points apart of distance $D$ and whose middle is the center of the ball.}
In other words, any convex set  in $\R^n$ with volume $V$, diameter $D$ and inradius $r$ satisfies:
\begin{equation}\label{completemin}
V\geq 2\omega_{n-1} r^n \int_{\arccos(2r/D)}^{\pi/2} \sin^nt dt+
\frac{\omega_{n-1} r^{n-1}}{nD^n}\left(D^2-4r^2\right)^{(n+1)/2}
\end{equation}
where $\omega_{n-1}$ is the volume of the unit ball in dimension $n-1$.
In particular, any convex set  in $\R^2$ with area $A$, diameter $D$ and inradius $r$ satisfies:
\begin{equation}\label{completemin2D}
A\geq r\sqrt{D^2-4r^2}+r^2\left(\pi-2\arccos\left(\frac{2r}{D}\right)\right).
\end{equation}
\end{theorem}

Let us turn to the maximization Problem \eqref{SOP}. Let us introduce particular convex sets of $\mathcal{K}^{n}_{r,D}$ that will be shown to be natural candidates to solve the maximization problem.

\begin{definition}[The symmetric spherical slice $K_{S}(D)$]\label{def1}

Let $D> 2$. We call symmetric spherical slice and denote by $K_{S}(D)$ the convex set defined as the intersection of the disc $D(O,D/2)$ with a strip of width $2$ centered at $O$ (see Fig. \ref{Def2}).
We have 
$$
\vert K_{S}(D)\vert=  \sqrt{D^{2}-4}+ \frac{D^{2}}{2}\arcsin\left(\frac{2}{D}\right).
$$
\end{definition}
\begin{figure}[h!]
\begin{center}
\includegraphics[width=6cm]{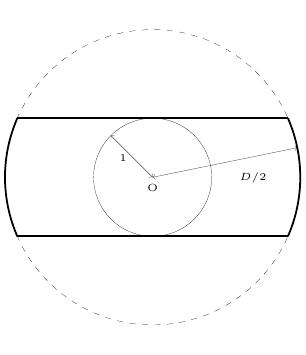}
\caption{The symmetric slice $K_{S}(D)$ and its (non unique) incircle.\label{Def2}}
\end{center}
\end{figure}

\begin{definition}[The smoothed regular nonagon $K_{E}(D)$]\label{defKE}
Let $D\in]2,2\sqrt{3}[$. We denote by $K_{E}(D)$ the convex set enclosed in an equilateral triangle $\Delta_E$ of inradius $1$ and made of segments and arcs of circle of diameter $D$ in
the following way (see Fig.~\ref{figTD2}): 
let $\eta_{i}$ be the normal angles to the sides of $\Delta_E$ (where one 
sets for example $\eta_{1}=-\pi/2$). Let us introduce
$$
\tau=(3+\sqrt{D^{2}-3})/2 \quad\text{and}\quad h=\sqrt{D^{2}-\tau^{2}}
$$
and the points $A_i$, $B_i$ and $M_i$, $i=1,2,3$ defined through their coordinates by
$$A_{i}=
\begin{pmatrix}
\cos\eta_i+h\sin\eta_i\\
\sin\eta_i-h\cos\eta_i
\end{pmatrix},
\quad B_{i}=
\begin{pmatrix}
\cos\eta_i-h\sin\eta_i\\
\sin\eta_i+h\cos\eta_i
\end{pmatrix}, \quad
M_{i}=(1-\tau)\times
\begin{pmatrix}
\cos\eta_i\\
\sin\eta_i
\end{pmatrix},\  i=1,2,3.
$$

The set $K_E(D)$ is then obtained as follows:
\begin{itemize}
\item the points $A_1$, $B_1$, $M_3$, $A_2$, $B_2$, $M_1$, $A_3$, $B_3$, $M_2$, $A_1 $ belong to its boundary;  
\item $\arc{B_{1}M_{3}}$ and $\arc{M_{1}A_{3}}$  are diametrally opposed arcs of 
the same circle of diameter $D$, and similarly for the two other pairs of 
arcs of circle $\arc{B_{2}M_{1}}$ and $\arc{M_{2}A_{1}}$, $\arc{M_2B_{3}}$ and $\arc{M_{3}A_{2}}$. 
\item the boundary contains the segment $[A_{i}B_{i}]$, $i=1,2,3$. Note 
that the contact point $I_{i}$ with the incircle is precisely the middle of $[A_{i}B_{i}]$, 
\end{itemize}
Moreover, setting
$$
t_{1}=\arccos\left(\frac{\sqrt{3}}{D}\right)=\arcsin\left(\frac{2\tau-3}{D}\right),\quad
t_{2}=\arccos\left(\frac{\sqrt{3}(\tau-2)}{D}\right)=\arcsin\left(\frac{\tau}{D}\right),
$$
 %Then $$\vert K_{a}(D)\vert = 3/4\times (D^{2}(t2-t1)-\sqrt{3}(\tau-2) 
%D(\sin(t2)-\sin(t1))+D(\tau-2)(\cos(t1)-\cos(t2))+4h)$$
one has 
\begin{equation}\label{formulaKE}
\vert K_{E}(D)\vert = \frac{3}{4} D^{2}(t_2-t_1)+\frac{3\sqrt{3}}{2}(\sqrt{D^{2}-3}-1)=
 \frac{3}{2} D^{2}\left(\frac{\pi}{3}-t_1\right)+\frac{3\sqrt{3}}{2}\left(\sqrt{D^{2}-3}-1\right).
\end{equation}

\begin{figure}[H]
\begin{center}
\includegraphics[width=8cm]{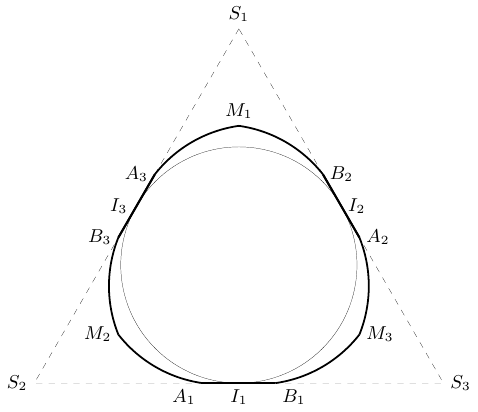}
\caption{The set $K_{E}(D)$ and its incircle\label{figTD2}}
\end{center}
\end{figure}

\end{definition}

In a nutshell, we will prove that {for $r=1$} the set  $K_{E}(D)$ is optimal for small values of $D$ whereas the solution is the symmetric slice for bigger values of $D$. {In what follows, the notation $rK$ with $r>0$ and $K\in  \mathcal{K}^{2}_{r,D}$ denotes the range of $K$ by the homothety  centered at the origin of the considered orthonormal basis, with scale factor $r$.}

\begin{theorem}\label{theo:SOmax}
{Let $r>0$ There exists $D^\star\simeq 2.3888$ such that if $D< rD^\star $, the (unique) solution of Problem~\eqref{SOP} is $rK_{E}(D/r)$, and for $D> rD^\star $ the unique solution is $rK_{S}(D/r)$.
For $D=D^\star r$ the two solutions coexist.}

In other words, for every plane convex set with area $A$, diameter $D$ and inradius $r$, one has
\begin{equation}\label{completemax}
A\leq \psi (D,r)\quad \text{where}\quad 
\psi(D,r)=\left\{\begin{array}{lc}
\frac{3\sqrt{3}r}{2}\left(\sqrt{D^2-3r^2} - r\right) + \frac{3 D^2}{2}\left(\frac{\pi}{3}-
\arccos\left(\frac{\sqrt{3} r}{D}\right)\right) & \mbox{if } D\leq r D^\star\\
r\sqrt{D^2-4} + \frac{D^2}{2} \arcsin\left(\frac{2r}{D}\right) & \mbox{if 
} D\geq r 
D^\star.

\end{array}\right.
\end{equation}

{More precisely $D^\star$ is the unique number in $[2,2\sqrt{3}]$ for which both expressions of $\psi(D,r)$ above are equal.}
\end{theorem}

\subsection{The Blaschke-Santal\'o Diagram for $(A,D,r)$}\label{sectionBS}
Usually, Blaschke-Santal\'o diagrams are normalized to fit into the unit square $[0,1]\times [0,1]$.
Thus, starting from the straightforward inequalities $D\geq 2r$ and $A\geq \pi r^2$ (where $A$, $D$ and $r$ denote respectively the area, diameter 
and inradius of any two-dimensional convex set), drives us to choose the
system of coordinates $x=2r/D$ and $y=\pi r^2/A$. We then define the Blaschke-Santal\'o diagram $\mathcal{D}$
as the set of points
$$
\mathcal{D}=\left\{(x,y)\in \mathbb{R}^2, x=2\frac{r(K)}{D(K)}, \ y=\pi \frac{r^2(K)}{A(K)}, \
K\in \mathcal K_2\right\}.
$$
The point $(1,1)$ corresponds to the disk, while the point $(0,0)$ corresponds to an infinite strip.
The solution of the minimization problem \eqref{SOP2}
%(of the area with a diameter and inradius constraints)
provided in Theorem~\ref{theo:SOmin} leads to the upper curve of $\mathcal{D}$.
Using \eqref{completemin2D}, we claim that the upper curve is the graph of $y^+$, defined by
$$y^+(x)=\frac{\pi x}{x(\pi-2\arccos x)+2\sqrt{1-x^2}}, \quad x\in [0,1].$$
According to Theorem~\ref{theo:SOmax}, the lower curve is the graph of $y^-$, piecewisely defined by
$$
y^-(x)=\left\lbrace 
\begin{array}{lc}
\displaystyle
\frac{\pi x}{2\sqrt{1-x^2}+2\frac{\arcsin x}{x}} & \mbox{if } x\leq 2/D^\star\\
\displaystyle
\frac{\pi x^2}{2\pi -6\arccos(\frac{\sqrt{3} x}{2})+\frac{3\sqrt{3} x}{2}\left(
\sqrt{4-3x^2} -x\right)} & \mbox{if } x\geq 2/D^\star.\\
\end{array}\right.
$$
Were already known the inequalities
\begin{itemize}
\item $4A\leq \pi D^2$ (see \cite{YBol}) which corresponds to the inequality $y\geq x^2$ on the diagram,
\item $A\leq 2rD$ (see \cite{HeTsi}) which is equivalent to $y\geq \frac{\pi x}{4}$ on the diagram.
\end{itemize}
These two inequalities are shown with a dotted line on the diagram hereafter.

\medskip

To plot the Blaschke-Santal\'o diagram, it remains to prove that the whole zone between the two graphs $\{(x,y^-(x)), \ x\in [0,1]\}$ and 
$\{(x,y^+(x)), \ x\in [0,1]\}$ is filled, meaning that each point between 
these two graphs corresponds to at least one plane convex domain.

Let us start with the part of the diagram on the left of $x\leq x^\star:=2/D^\star$.
For a given diameter $D$ and inradius $r$, let $K^-$ denote the convex set with minimal area
(the two-cap body) and $K^+$ the convex set with maximal area (the symmetric slice).
We have $K^-\subset K^+$ and for any $t\in [0,1]$ the convex set $K_t:$ constructed according to the Minkowski sum $K_t=t K^+ + (1-t)K^-$ with $t\in [0,1]$, is known to
satisfy $K^-\subset K_t\subset K^+$. Therefore, all the sets $K_t$ share the same diameter $D$, the same inradius $r$ and their area is increasing 
from $A(K^-)$ to $A(K^+)$. This way,
it follows that the whole vertical  joining $(2r/D,y^-(2r/d))$ to $(2r/D,y^+(2r/d))$ is included in 
$\mathcal{D}$ as soon as $2r/D\leq 2/D^\star$.

Let us consider the remaining case $x\geq x^\star:=2/D^\star$. Starting 
from the optimal domain
$K^+$ which maximizes the area with given $D$ and $r$ (recall that $K^+$ is the convex set inscribed
in the equilateral triangle introduced in Definition~\ref{defKE}), we fix 
one of its diameter, say $[A,B]$
and we shrink continuously $K^+$ to the set $K_{AB}$ defined as the convex hull of the
points $A,B$ and the disk of radius $r$ contained in $K^+$. {Secondly}, 
we move the points $A,B$ continuously to the points $A',B'$ at distance $D$, oppositely located
with respect to the center of the disk (in the sense that the center is the middle of
$A',B'$) by keeping the convex hull with the disk at each step. The final 
step
is therefore the two-cap body $K^-$ and we have constructed a continuous path between 
$K^+$ and $K^-$ keeping the diameter and the inradius fixed: it follows that the whole  
 joining $(2r/D,y^-(2r/d))$ to $(2r/D,y^+(2r/d))$ for $2r/D\geq 2/D^\star$ is included in 
$\mathcal{D}$. At the end, $\mathcal{D}$ has only one connected component.

The complete Blaschke-Santal\'o diagram is plotted on Fig.~\ref{fig:BS} below.

\begin{figure}[H]
\begin{center}
\includegraphics[width=10cm,]{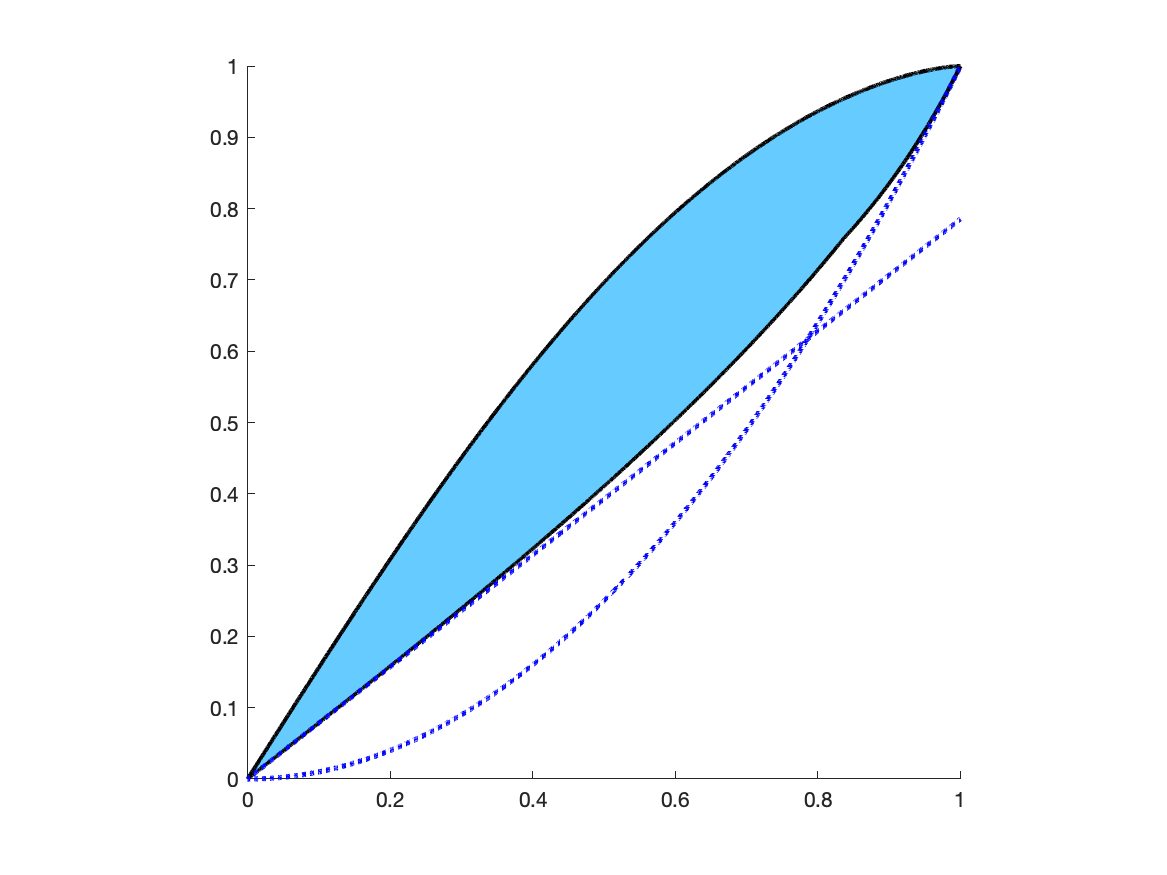}
\caption{The Blaschke-Santal\'o diagram $\mathcal D$ for $(A,D,r)$ (colored picture). The dotted lines represents the known inequalities $4A\leq \pi D^2$ and $A\leq 2rD$. \label{fig:BS}}
\end{center}
\end{figure}

\begin{remark}
It is notable that the two-cap body has been showed to solve a shape optimization problem motivated by the understanding of branchiopods eggs geometry in biology, and involving packings (see \cite{MR3941924}).
\end{remark}

%%%%%%%%%%%%%%%%%%%%%%%%%%%%%%%%%%%%%%%%%%%%%%%%%%%%%
%%%%%%%%%%%%%%%%%%%%%%%%%%%%%%%%%%%%%%%%%%%%%%%%%%%%%

\section{Proof of Theorem \ref{theo:SOmin}}\label{sec:proofTheoMin}

Let us first introduce several notations. For a generic  convex set $K$, we will denote by $A$ and $B$ the points of $K$ realizing the diameter, and respectively by $O$ and $r$ the center and radius of an insphere (the boundary of the biggest ball included in $K$). Introduce  $\mathcal{B}=(e_{1},...,e_{n})$ an orthonormal basis such that $e_{n}=\overrightarrow{AB}/AB$, so that the coordinates of $A$ and $B$ in $\mathcal B$ are
$$
A=(0,0,...,0)\quad \text{ and }\quad B=(0,0,...,0,D).
$$
More generally, we will denote by $(x_1,\dots,x_n)$ the coordinates of a generic vector $X$ in $\mathcal{B}$.

First, in order to relax the conditions $D(K)= D$ and $r(K)=r$ in Problem \eqref{SOP2}, we show that it is equivalent to deal with the conditions $r(K)\geq r$ and $D(K)\geq D$, which are always saturated at the optimum.
 
%In what follows, we will consider $K^\star$, a solution of Problem \eqref{SOP2}. 

 \begin{lemma}\label{train0818}
 Let $r>0$ and $D>2r$. Let us consider the minimization problem
\begin{equation}\label{SOP2bis}\tag{$\widehat{\mathcal{P}}_{\rm min}$}
\inf_{K\in \widehat{\mathcal{K}}^{n}_{r,D}}\vert K \vert.
\end{equation}
 where
 $
 \widehat{\mathcal{K}}^{n}_{r,D}=\{K\in \mathcal K_n \mid r(K)\geq r\text{ and } D(K)\geq D\}.
 $
 Then, Problem~\eqref{SOP2bis} has at least a solution $K^\star$ and moreover, one has $D(K^{\star})=D$ and $r(K^\star)=r$.
 \end{lemma}

\begin{proof}
Existence of $K^\star$ follows by an immediate adaptation of the proof of 
Proposition~\ref{prop:exist} (if the diameter goes to $+\infty$ it is easy to prove that the
volume must blow up).

Regarding the second part of the statement, let us argue by contradiction,  assuming that $r(K^\star)>r$. We use the coordinate system associated to the basis $\mathcal{B}$ introduced above, constructed from a diameter $[AB]$ of $K^\star$.
Defining $\lambda=r/r(K^\star)<1$ and applying to $K^\star$ the linear transformation whose matrix in $\mathcal B$ is  $\operatorname{diag}(\lambda,...,\lambda,1)$, we obtain a new convex set $K'$ with diameter $D$ and inradius $r$. Moreover, its volume is $\lambda^{n-1}\vert K^\star \vert<\vert K^\star\vert$. this is in contradiction with the minimality of $K^\star$. 

Similarly, arguing still by contradiction, let us assume that $D(K^\star)>D$. Since $D>2r=2r(K^\star)$, there exist $A'$ and $B'$ in $[A,B]$ such that $A'B'=D$. Given $O$, the center of an insphere, we consider the set $K'$ defined as the convex hull of $A'$, $B'$ and $B(O,r)$. From this 
construction and by convexity, $K'$ is strictly included in $K$, $D(K')\geq D$ and $r(K')\geq r$. Therefore, one has $K'\in  \widehat{\mathcal{K}}^{n}_{r,D}$ and $\vert K'\vert<\vert K \vert$, which is in contradiction with the optimality of $K$. The conclusion follows. 
\end{proof}

It follows in particular from this result that the solutions of Problems~\eqref{SOP2} and \eqref{SOP2bis} coincide.

Furthermore, if $K$ is a general convex body in $\mathcal{K}^{n}_{r,D}$, by repeating the argument used to deal with the diameter constraint in the proof of Lemma~\ref{train0818}, one sees that the convex hull of $A$, $B$ and $B(O,r)$ also belongs to $\mathcal{K}^{n}_{r,D}$ and has a lower measure than the one of $K$.

Therefore, any minimizer $K^\star$ is necessarily the convex hull of two points $A$ and $B$ realizing its diameter, and $B(O,r)$, whose boundary is an insphere We note $K_O$ such a set. The next result proves a symmetry 
property of $K^\star$.

\begin{lemma}\label{lem:train1848}
{Let $D >2r>0$ and $A,B$ be two points at distance $D$
in $\R^{n}$. For any $O\in \R^{n}$, define the set $K_O := \operatorname{conv}(A, B, B(O, r))$. Then
$K_O\in \mathcal{K}^n$ and $\vert K_O\vert  \geq \vert K_{O'}\vert$ where 
$O'$ is the orthogonal projection of $O$
onto the line containing $A$ and $B$, with equality if and only if $O = 
O'$.}
%Let us denote by $x_O\in [0,D]$ the first coordinate of $O$ in the basis 
%$\mathcal B$ introduced above. Let us introduce $K'$ as the convex hull of $A$, $B$ and $B(O',r)$ where the coordinates of $O'$ in $\mathcal B$ are $(x_O,0,..,0)$ (in other words, $O'$ is the orthogonal projection of $O$ on the axis $(A;e_{1})$). then $K'$ belongs to $\mathcal{K}^{n}_{r,D}$ and there holds $\vert K'\vert\leq  \vert K^{\star} \vert$ with equality if, and only if $O=O'$.
\end{lemma}

\begin{proof}
Assume that that $O\neq O'$ we will prove that $\vert K_{O}\vert >\vert K_{O'}\vert $.
Two cases may happen.

\begin{enumerate}
\item
The ball $B(O,r)$ does not meet the diameter $[AB]$.
\item
The ball meets the diameter $[AB]$.
\end{enumerate}

In the first case let $a=OO'-r>0$, and assume that $e_1=\overrightarrow{{OO'}}/OO'$. 
Let us consider $S(K_O)$ the Steiner symmetrization of $K_O$ with respect 
to the hyperplane with normal vector $e_1$ and containing $A$ and $B$.
It is a well known result (see \cite{bonnesen}) that $S(K_O)$ is still convex with same area as $K_O$. Furthermore it contains $B(O',R),A$ and $B$. So it contains $K_{O'}$. Let us finally remark that $K_O\cap (OO')$ has length $2r+a$, and so $S(K_O)$ contains the point $C=(x_O,r+a/2,0,..,0)$ which is not in $K_{O'}$. By convexity we deduce that  $\vert K_{O'}\vert<\vert K_O\vert$.

\vspace{3mm}
 In the second case, we will distinguish three parts in $S(K_O)$, and for 
each part we will compare the volume of $K_O$ with the one of $K_{O'}$. {The main difficulty of what follows consists in proving that the area of the set $K_{O'}$ is strictly smaller than that of $K_O$, the corresponding large inequality being easily obtained with the properties of the symmetrization. }
 Consider the upper part $K_O^+$ of $K_O$, namely $K_O\cap \{X\in \R^n\mid X\cdot e_n\in [x_O,D]\}$.
Let $\Gamma_{B}$ be the set of points of $B(O,r)$ whose tangent hyperplane contains $B$, and $\Gamma'_{B}$ be the set of points of $S(O',r)$ whose 
tangent hyperplane contains $B$. By symmetry, all the points of $\Gamma'_{B}$ share the same last coordinate $x'$. Let $x_{1}$ and $x_{2}$ denote respectively the minimal and maximal first coordinate of points of $\Gamma_B$. Hence, one has $x_O+r>x_{2}>x_1>x_O$ and moreover, $x'\in (x_{1},x_{2})$ (see points $M,M_1,$ and $M_2$ in Fig. \ref{fig:lem2}).
%$M_{1}$ and $M_{2}$ den $\Gamma_{B}$ with minimal (resp. maximal) first coordinate $x_{1}$ (resp. $x_{2}$) with $x+r>x_{2}>x_{1}>x$. We claim that $x'\in ]x_{1},x_{2}[$. 

Let us distinguish between three zones of $K_O^{+}$:
\begin{itemize}
\item \textit{On $K_O^+\cap \{X\in \R^n\mid X\cdot e_n\in [x_O,x_1]\}$.} It is easy to see that $B(O',r)\cap \{X\in \R^n\mid X\cdot e_n\in [x_{O},x_1]$ is exactly the image of 
$$
K_O^+\cap \{X\in \R^n\mid X\cdot e_n\in [x_O,x_1]\} =B(O',r)\cap \{X\in 
\R^n\mid X\cdot e_n \in [x_{O},x_1]\}.
$$ 
by the translation vector $\overrightarrow{O'O}$. These two sets have therefore the same measure.
\item \textit{On $K_O^+\cap \{x\in [x_1,x']\}$.} For $x\in \R$, let $H_{x}$ be the affine hyperplane whose equation in $\mathcal B$ is $\{X\in \R^n\mid X\cdot e_n=x\}$, and introduce  $K_{x}=K_O\cap H_x$. 
If $x\in [x_{O}-r,x_{O}+r]$, let $B_{x}$ be the $n-1$ dimensional ball $B(O',r)\cap H_{x}$. By construction, one has $\h(B_{x})<\h(K_{x})$ for all 
$x>x_{1}$ . As a consequence 
$$
|B(O,r)\cap \{X\in \R^n\mid X\cdot e_n\in [x_{1},x']\}|<|K \cap \{X\in \R^n\mid X\cdot e_n \in [x_{1},x']\}|.
$$
\item \textit{On $K_O^+\cap \{X\in \R^n\mid X\cdot e_n\in [x',D]\}$.} Define $C_{x'}$ as the cone with vertex $B$ and basis $B_{x'}=B(O',r)\cap H_{x'}$. Since $C_{x'}$ is the convex hull of $B_{x'}$ and $B$, it follows that $|C_{x'}|<|K^\star\cap \{X\in \R^n\mid X\cdot e_n\in [x',D]\}|$.
%$\mathcal{S}(C_{x'})=K'\cap \{x\in [x',D]\}$. 
\end{itemize}
It follows that $\vert K_{O'} \cap \{x\in [x_{0},D]\} \vert< \vert K_O \cap \{x\in [x_{0},D]\} \vert$. Doing the same construction on the lower part of $K_O$ yields at the end that $\vert K_{O'}\vert<\vert K_O \vert$. The expected result follows.
\end{proof}

\begin{figure}[H]
\begin{center}
\includegraphics[width=4cm]{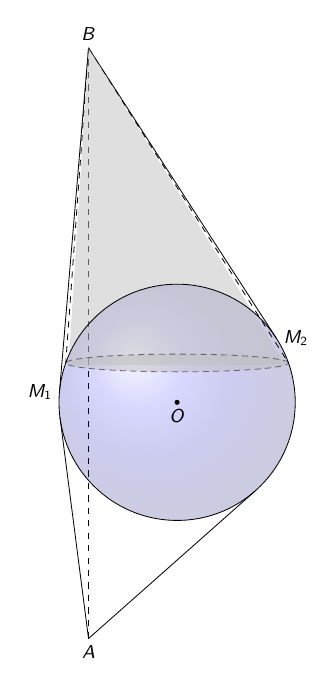}\hspace{2cm}
\includegraphics[width=3.7cm]{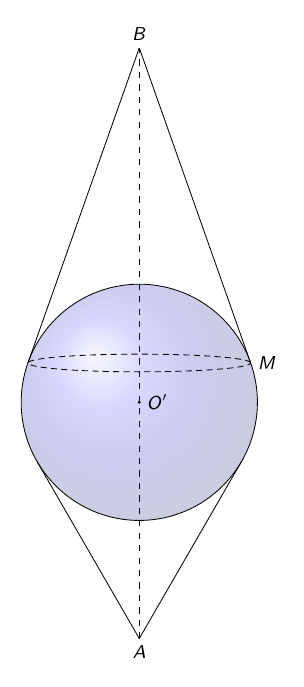}
\caption{Illustration of the proof of Lemma~\ref{lem:train1848}. The convex set on the right
has the same inradius and diameter as the one on the left but a lower volume.\label{fig:lem2}}
\end{center}
\end{figure}

To sum-up, we know that any minimizer $K^\star$ is of the type $K_O$, the 
convex hull of $A$, $B$ and $B(O,r)$, where $AB=D$ and $A$, $B$ and $O$ 
are collinear. it remains to show that the minimum is reached whenever $O$ is in the middle of the  $[AB]$. This can be done by an explicit
computation, but we propose a more geometrical proof based again on Steiner symmetrization.
 
 Let us argue by contradiction, considering $O\in [AB]\backslash \{ I\}$, 
where $I$ is the middle of $[AB]$ and assuming that $K^\star=K_O$. Let $\mathcal H$ be the hyperplane containing $I$ with normal vector $\overrightarrow{AB}$. Let $K'$ be the Steiner symmetrized of $K^\star$ with respect to $\mathcal H$. We claim that $K'\in \mathcal{K}^{n}_{r,D}$. Indeed, 
by monotonicity of the Steiner symmetrization with respect to the inclusion and since the range of $B(O,r)$ by the Steiner symmetrization is $B(I,r)$, one has necessarily $r(K')\geq r(K^\star)$. In the same way, observe 
that the strip $\mathcal S:=\{x\in \R^n\mid x_1\in [-r/2,r/2]\}$ is invariant by the Steiner symmetrization and contains $K^\star$. By using again the aforementioned monotonicity property, one has also $K'\subset \mathcal S$, and therefore, $r(K')\leq r=r(\mathcal S)$. Therefore, one has 
$r(K')=r$. It is standard that Steiner symmetrization reduces diameter. 
Moreover, since $[AB]$ is invariant by the Steiner symmetrization and since $[AB]\subset K'$, one has $D(K')\geq D$  and thus $D(K')=D$.
 
 Since $|K'|=|K^\star|$ by property of the Steiner symmetrization, it follows that $K'$ solves Problem~\eqref{SOP2}.

{It now remains to investigate the equality case, namely to compare $|K'|$ and $|K_I|$ where we recall that $K_I=\operatorname{hull}(A,B,B(I,r))$.  More precisely we will prove that $ K'$ has a larger volume than $K_I$.} In the basis $\mathcal{B}$, let $x_1^*\in (0,r)$ be such that
$
K_I\cap \{x_1\geq x_1^\star\}=B(I,r)\cap \{x_1\geq x_1^\star\}$
and $B(O,r)\cap \{x_1\geq x_1^\star\}\subsetneq K_O\cap \{x_1\geq x_1^\star\}.
$
The existence of $x_1^\star$ follows from the dissymmetry of $K_O$ with respect to $\mathcal{H}$. Using one more time the monotonicity property of 
the Steiner symmetrization with respect to the inclusion, one has 
$$
B(I,r)\cap \{x_1\geq x_1^\star\}\subsetneq K'\cap \{x_1\geq x_1^\star\},
$$
which implies that the volume of $K'$ is strictly larger than the one of $\operatorname{hull}(A,B,B(I,r))$. We have thus reached a contradiction and it follows that one has necessarily $O=I$, meaning that
$K^\star=\operatorname{hull}(A,B,B(I,r))$, which concludes the proof.

%%%%%%%%%%%%%%%%%%%%%%%%%%%%%%%%%%%%%%%%%%%%%%%%%%%%%
%%%%%%%%%%%%%%%%%%%%%%%%%%%%%%%%%%%%%%%%%%%%%%%%%%%%%

\section{Proof of Theorem \ref{theo:SOmax}}\label{sec:proofTheoMax}

In the whole proof, for a given set $K\in \mathcal K_2$, we will denote by $C_K$ an incircle of $K$. It is standard that $K$ is tangent to $C_K$ at two points at least. 

\begin{definition}
Let $K\in \mathcal K_2$. A point $x \in K$ is said to be diametral if there exists $y\in K$ such that $\Vert x-y\Vert=D(K)$.
\end{definition}

Obviously, if $x$ is diametral, then it belongs necessarily to $\partial K$. Denoting by $y$ its counterpart, if the boundary of $K$ is $\mathscr{C}^1$ at $x$, the outward unit normal vector at $x$ on $\partial K$ is $n(x)= (x-y)/\Vert x-y\Vert$.

In what follows, we will consider a solution $K^\star$ to Problem~\eqref{SOP}, whose existence is provided by Proposition~\ref{prop:exist}.

Since the area is maximized, it seems natural to look for the largest possible set and thus to saturate the diameter constraint at each point. Nevertheless, the inradius constraint tends to stick the convex body onto the circle. M.~Belloni and E.~Oudet in \cite{BelOud} worked on the minimal gap between the first eigenvalue of the Laplacian $\lambda_2$ and the first eigenvalue of the $\infty-$Laplacian $\lambda_\infty$. Since $\lambda_\infty(\Omega)=1/r(\Omega)$ and $\lambda_2$ is decreasing for the inclusion, some of their results were obtained by constructing bigger sets while maintaining the inradius and the diameter. The following lemma is an example.
\begin{lemma}[\cite{BelOud}]\label{lem1}
Let $x\in \partial K^{\star}$. Then, one has the following alternative:
\begin{enumerate}
\item $x$ is non diametral and belongs to the interior of a segment of $\partial K^{\star}$.
\item $x$ is diametral and is not in the interior of a segment of $\partial K^{\star}$.
\item $x$ is in the intersection of two segments of $\partial K^{\star}$.
\end{enumerate}
\end{lemma}

To locate the segments of $\partial K^\star$ and provide an estimate of their numbers, we need the notion of {\it contact point}.

\begin{definition}\label{def:contact}
A contact point of $\partial K^\star$ is a point $x$ at the intersection of $\partial K^{\star}$ and an incircle $ C_{K^{\star}}$ of $K^\star$. Similarly, a contact line is a support line of $K^{\star}$ passing by a contact point. Note that it is also a support line of $C_{K^{\star}}$.
\end{definition}

Observe that the relative interior of a {segment of $\partial K^{\star}$} is necessarily made of non diametral points. 

Note that the incircle is \textit{a priori} not unique. Let us consider all the possibilities:
%, and it could be problematic to talk about a contact line when there could be at least two incircle. In the following we show that this is not really a problem.
\begin{itemize}
\item {\bf case 1: }the incircle is not unique. In that case the convex $K^{\star}$ is necessarily included in a strip of width $2$, and every incircle touches both lines of the strip.

Indeed, let $C_{1}$ and $C_{2}$ be two incircle and $O_{1}$ and $O_{2}$ their center. We consider a basis in which the coordinates of $O_{1}$ are $(-a,0)$ and those of $O_{2}$ are $(a,0)$. Let $N_{i}$ (resp $S_{i}$) be the north (resp. south) pole of $C_{i}$. By convexity the rectangle $N_{1}N_{2}S_{2}S_{1}$ is included in $K^{\star}$. Now suppose that $K^{\star}$ is not included in the strip formed by the lines $(N_{1}N_{2})$ and $(S_{1}S_{2})$. Then there exist a point $M(x,y)\in K^{\star}$ with $-a\leq x\leq a$ and $y>1$. By construction, the pentagon $N_{1}MN_{2}S_{2}S_{1}$ 
is convex, included in $K^{\star}$, and its inradius is larger than $1$ (see Fig.~\ref{fig:strip}) which contradicts the inradius constraint.

\begin{figure}[H]
\begin{center}
\includegraphics[width=10cm,]{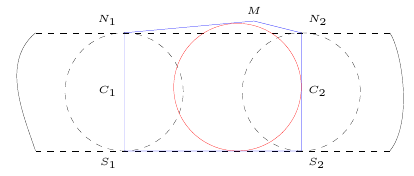}
\caption{The middle circle is larger than the others, so the inradius is larger than 1.\label{fig:strip}}
\end{center}
\end{figure}
%
%A convex set with a non unique incircle is necessarily included in a strip (see Fig.~\ref{fig:stadedeforme}).
%\begin{figure}[H]
%\begin{center}
%\includegraphics[width=5cm, angle=90]{StadeDeforme.pdf}
%\caption{A convex set whose (non unique) incircle has two parallel contact line. \label{fig:stadedeforme}}
%\end{center}
%\end{figure}

\item {\bf case 1bis: }the incircle is unique, but still inscribed between two strips. In this case it is even included in a square, which is covered by the case $1$.

\item {\bf case 2 : }the incircle is unique, and there are exactly three contact lines, forming a triangle containing both the circle and the convex.

\begin{figure}[H]
\begin{center}
\includegraphics[width=6cm]{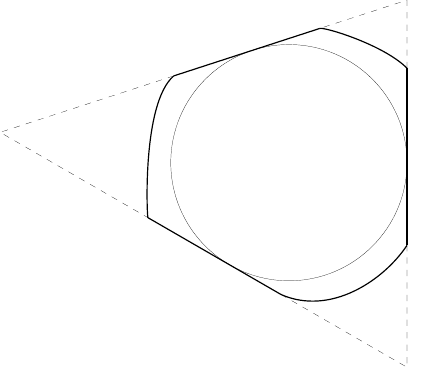}
\caption{A convex set with three contact points \label{fig:3cotes}}
\end{center}
\end{figure}
\end{itemize}
%In the first case, we can fix arbitrarily any incircle to work with, since the contact lines are the same. It allows us to talk about a convex body $K$ and its incircle $C_{K}$ without risk.

We sum-up these information in the following lemma.

\begin{lemma}\label{lem2}
Any  segment of $\partial K^{\star}$ contains a contact point. Furthermore, $\partial K^{\star}$ contains at most three segments.
\end{lemma}

\begin{proof}
If a  segment of  $\partial K^{\star}$ does not touch an incircle, it would be possible to inflate this part without changing the inradius nor violating the diameter constraint. The upper bound on the number of  segments is a direct consequence of the previous analysis: if $K$ has more than the minimal numbers of segments that are useful to prescribe the incircle, then some are useless and can be inflated without consequences on the constraints. \end{proof}

In what follows, we will work separately on the cases 1 and 2. Section~\ref{strip} deals with the first case, whereas Section~\ref{triangle} is devoted to the investigation of the second case.

Thanks to an easy renormalization argument, we will assume without loss of generality that the inradius of the considered convex sets is equal to 1 ($r=1$).

%Before treating separately each case, let us first provide a useful lemma. Denoting in the case 1 by $D_{1}$ and $D_{2}$ or by $D_{1}$, $D_{2}$ and $D_{3}$ in the case 2, the support lines common to the incircle $C_K$ and $K$ leads to decomposing $\partial K$ as $\partial K=K_{\text{fix}}\cup K_{\text{free}}$ where $K_{\text{fix}}=(\partial K\cap \cup_{i}D_{i})\cup \partial K \backslash$ and $K_{\text{free}}= \cup_{i}D_{i}$. The first set $K_{\text{fix}}$ is the fixed part of $\partial K$, whereas $K_{\text{free}}$ is the free part of $\partial K$, an open subset of $\partial K$.
%
%\begin{lemma}\label{L1}
%Let $K^{\star}$ be a solution of \eqref{SOP}. We write $\partial K^*=K_{\text{fix}}^*\cup K_{\text{free}}^*$, where $K_{\text{fix}}$ and $K_{\text{free}}$ are defined as previously.
% Then, for every $x\in K_{\text{free}}^*$, there exists $y\in K$ such that $\operatorname{dist}(x,y)=\operatorname{Diam}(K^{\star})$. In other words, the diameter constraint is saturated in the free portion.
%\end{lemma}
%
%
%\begin{proof}
%Assume the existence of $x\in K_{\text{free}}^*$ such that $\sup_{y\in K^{\star}}\operatorname{dist}(x,y)(x,y)< D$. Then, by continuity of the distance, there exist $\e>0$ and $\delta>0$ such that $D(x,\delta)\cap \partial K\subset K_{\text{free}}^*$ and for all $x'\in D(x,\delta)\cap \partial K$, where $D(x,\delta)$ is the open disk centered at $x$ with radius $\delta$, one has $\sup_{y\in K^{\star}}\operatorname{dist}(x,y)\leq D-\e$. Define $r=\min(\delta/2, \e/2)$ and $K=K^{\star}\cup D(x,r)$. Then 
$\operatorname{Diam}(K)=D$ but $\vert K \vert > \vert K^{\star}\vert $, 
leading to a contradiction with the optimality of $K^{\star}$.
%
%
%\end{proof}
\subsection{First case: $K^\star$ is included in a strip}\label{strip}

Let $C_{K^\star}$ be an incircle of $K^\star$. To investigate the case where $K^\star$ is included in a strip, we consider a basis $\mathcal B$ whose origin $O$ is the center of $C_{K^\star}$ and such that the equations 
of the two contact points support lines are  $x=1$ and $x=-1$ (see Fig.~\ref{fig:stadiumcase1}). Let us denote by $\mathcal S$, the closed strip $\{(x,y)\in \R^2 \mid |x|\leq 1\}$.

We investigate in this section a constrained version of Problem~\eqref{SOP}, namely
\begin{equation}\label{SOP1}\tag{$\mathcal{P'}$}
\sup_{\substack{K\in \mathcal{K}_{r,D}^2\\ K\subset \mathcal S}}\vert K \vert.
\end{equation}

%by $\mathcal{K}^{1}_{1,D}$ the subset of $\mathcal{K}_{1,D}$ with such convex set. 

\begin{proposition}\label{prop:case1}
The symmetric slice $B(O,D/2)\cap \mathcal S$, where $B(O,D/2)$ denotes the open ball centered at $O$ with radius $D/2$, is the unique solution of 
Problem \eqref{SOP1}. The optimal area is 
$$
\max_{\substack{K\in \mathcal{K}_{r,D}^2\\ K\subset \mathcal S}}\vert K \vert =\sqrt{D^{2}-4}+\frac{D^{2}}{2}\arcsin\left(\frac{2}{D}\right).
$$
\end{proposition}

The set $K^\star$ is plotted on Fig~\ref{fig:stadiumcase1} right.
\begin{figure}[H]
\begin{center}
\begin{tabular}{cc}
\begin{minipage}{4cm}
\phantom{a}\\
\includegraphics[width=3cm, angle=0]{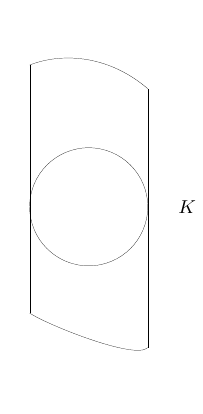} \hspace{5cm}
\end{minipage} & 
\begin{minipage}{4cm}
\includegraphics[width=3.2cm,angle=0]{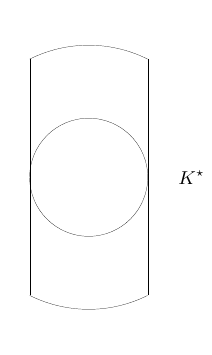}
\end{minipage}
\end{tabular}
\caption{Left: a convex set {$K$} whose (non unique) incircle has two parallel contact lines. Right: the optimal domain {$K^\star$} among convex sets included in a slice.\label{fig:stadiumcase1}}
\end{center}
\end{figure}
The end of this section is devoted to the proof of Prop.~\ref{prop:case1}. It is straightforward that, if a convex set $K$ belongs to $\mathcal{K}_{r,D}^2$ and is included in $\mathcal S$, then there exist two concave 
nonnegative functions $f$ and $g$ on $[-1,1]$ such that %$K=K_{f,g}$ where 
\begin{equation}\label{Kstarparam:fg}
K=\{(x,y)\in \R^{2}, x\in [-1,1], -g(x)\leq y\leq f(x)\}.
\end{equation}
With these notations, the optimal set $K^\star$ introduced in Prop.~\ref{prop:case1} corresponds to the choices
$$
f=y_D, \quad g=y_D\quad \text{ where }\quad y_D(x)=\sqrt{D^{2}/4-x^{2}}.
$$
%Notice that one has $K_{\text{free}}^*=f((-1,1))\cup -g((-1,1))$. 

The proof consists of two steps: first, we provide necessary optimality conditions on an optimal pair $(f,g)$ and show in particular that the aforementioned symmetric slice is a solution. Then, we investigate uniqueness 
properties of the optimum.

\begin{lemma}\label{lem3}
Let $K^{\star}$ be a solution of Problem~\eqref{SOP1}. Then, $K^\star$ is 
of the form \eqref{Kstarparam:fg} and satisfies
\begin{equation}\label{prop:strip}
f(x)+g(x)+f(-x)+g(-x)=4y_D(x), \quad x\in [-1,1].
\end{equation}
Furthermore, the convex set $\tilde K$ of the form  \eqref{Kstarparam:fg} 
with $f=g=y_D$ solves Problem \eqref{SOP1}.
\end{lemma}

\begin{proof}
We already know that $K^{\star}$ writes as \eqref{Kstarparam:fg} for some 
positive concave functions $f$ and $g$.

First, by lemma \ref{lem1}, every point of the free boundary part $\partial K^\star_{\rm free}:=\partial K^\star \cap \{(x,y)\in \R^2\mid x\in (-1,1)\}$ is necessarily diametral. As a consequence, the functions $f$ and $g$ are strictly concave. Indeed, observe that a segment of the boundary of a convex set contains at most two diametral points. 
 
%By  the triangular inequality, if $x\in [-r,r]$ and $(x,f(x)) \in K_{\text{free}}^*$, then there exists $x'\in [-1,1]$ such that $D=\operatorname{dist}((x,f(x)),(x',-g(x')))$
From the parametrization of $K^\star$, we get
\begin{equation}
D^{2}=\max_{(x,x')\in [-1,1]^2} (x-x')^{2}+(f(x)+g(x'))^{2}
\quad \text{and}\quad \vert K^{\star}\vert=\int_{-1}^{1}(f+g).
\end{equation}

We are going to prove the result by performing two consecutive Steiner symmetrizations, the first in the horizontal axis, the second in the vertical axis. Note that those two particular symmetrizations do not change the 
inradius.

\medskip

Let us introduce the set $\hat K$ of the form \eqref{Kstarparam:fg} where 
$f$ and $g$ are both replaced by $(f+g)/2$. In other words, $\hat K$ is the Steiner symmetrized of $K^\star$ with respect to the horizontal axis. Hence, one gets easily that $\vert K^{\star} \vert =\vert \hat K\vert$, 
and $D(\hat K)\leq D(K^\star)$. %\begin{eqnarray*}
%D(\hat K)^{2}&=&\max_{(x,x')\in [-1,1]^2} (x-x')^{2}+((f(x)+g(x))/2+(f(x')+g(x'))/2))^{2}\\
%&=&\max_{(x,x')\in [-1,1]^2} (x-x')^{2}+((f(x)+g(x'))/2+(f(x')+g(x))/2))^{2}\\
%&\leq& %\max_{(x,x')\in [-1,1]^2} (x-x')^{2}+(f^{\star}(x)+g^{\star}(x'))^{2}/2+(f^{\star}(x')+g^{\star}(x))^{2}/2\\
%%&=&  
%\frac12\max_{(x,x')\in [-1,1]^2} (x-x')^{2}+(f(x)+g(x'))^{2}+\frac12 \max_{(x,x')\in [-1,1]^2} (x-x')^{2}+(f(x')+g(x))^{2}\\
%&=& \max_{(x,x')\in [-1,1]^2} (x-x')^{2}+(f(x)+g(x'))^{2}=D^{2}
%\end{eqnarray*}
%Let us distinguish between two cases. 
Moreover, if $D(\hat K)<D(K^\star)$, then $\hat K$ is a convex set having 
the same area as $K^\star$, but a strictly lower diameter. Mimicking the argument used in the proof of Lemma \ref{train0818} allows us to obtain a 
convex set in $\mathcal{K}^{2}_{1D}$ with a larger area than $K^{\star}$, 
which is impossible. It follows that one has necessarily $D(\hat K)=D(K^\star)$.

Let us set $f^{\star}=(f+g)/2$ and let $x\in [-1,1]$. 
%The point of coordinates $(x,f^\star (x))$ is diametral. Let $x'\in [-1,1]$ be such that $\Vert (x,f^\star(x))-(x',-f^\star(x'))\Vert=D$. We claim that one has necessarily $x'=-x$. Indeed, assume by contradiction that $x'\neq -x$. By strict concavity of $f$, one has
%\begin{eqnarray*}
%D^2&=& (x-x')^{2}+(f^\star(x)+f^\star(x'))^{2}\\
%&<& ((x-x')/2+(x-x')/2)^{2}+(f^\star((x-x')/2)+f^\star((x-x')/2))^{2}
%\end{eqnarray*}
%
%
%let $y=-x'$.
%$$
%D^2=(x+y)^{2}+(f^\star(x)+f^\star(y))^{2}<((x+y)/2+(x+y)/2)^{2}+(f^\star((x+y)/2)+f^\star((x+y)/2))^{2}
%$$
%and the diameter constraint cannot be satisfied. It follows that necessarily $x'=-x$.
%
%
%
%Setting $f^{\star}=(f+g)/2$, we claim that $f^\star$ is even. Indeed, 
Let $K_{\tilde f}$ be a set of the form \eqref{Kstarparam:fg} where $f$ and $g$ are both replaced by $\tilde f$ defined by
$$
\tilde f(x)= \frac{f^{\star}(x)+f^{\star}(-x)}{2}, \quad x\in [-1,1].
$$
{In other words, $K_{\tilde f}$ corresponds to the Steiner symmetrization of $\hat K$ with respect to the vertical axis. Then, using one more 
time standard properties of the Steiner symmetrization, one gets that $\vert K^{\star}\vert = \vert K_{\tilde f}\vert$ and for the same reasons as before, we have $D(K^\star)=D(K_{\tilde f})$.  Therefore, we have constructed a solution with two axes of symmetry. }
%We now investigate the equality case.}
%\begin{align*}
%D(K_{f^\star})^{2}=& \max_{(x,x)'\in [-1,1]^2} (x-x')^{2}+(\tilde f(x)+\tilde f(x'))^{2}\\
%%=& \max_{(x,x')\in [-1,1]^2} (x-x')^{2}+((f^{\star}(x)+f^{\star}(-x))/2+(f^{\star}(x')+f^{\star}(-x'))/2)^{2}\\
%=& \max_{(x,x')\in [-1,1]^2} (x-x')^{2}+((f^{\star}(x)+f^{\star}(x'))/2+(f^{\star}(-x)+f^{\star}(-x'))/2)^{2}\\
%\leq & \frac12\max_{(x,x')\in [-1,1]^2} (x-x')^{2}+(f^{\star}(x)+f^{\star}(x'))^{2}+ \frac12\max_{(x,x')\in [-1,1]^2} (-x+x')^{2}+(f^{\star}(-x)+f^{\star}(-x'))^{2}\\
%=&\max_{(x,x')\in [-1,1]^2} (x-x')^{2}+(f^{\star}(x)+g^{\star}(x'))^{2}=D^{2}
%\end{align*}

%As we have seen at the beginning of the proof, any point of coordinates $(x,\tilde{f} (x))$ is diametral. Let $x'\in [-1,1]$ be such that $\Vert (x,\tilde{f}(x))-(x',-\tilde{f}(x'))\Vert=D$. We claim that one has necessarily $x'=-x$. Indeed, assume by contradiction that $x'\neq -x$. By strict concavity of $\tilde{f}$, one has
%\begin{eqnarray*}
%D^2&=& (x-x')^{2}+(\tilde{f}(x)+\tilde{f}(x'))^{2}=(x-x')^{2}+(\tilde{f}(x)+\tilde{f}(-x'))^{2}\\
%&<& ((x-x')/2+(x-x')/2)^{2}+(\tilde{f}((x-x')/2)+\tilde{f}((x-x')/2))^{2}\\
%&=& ((x-x')/2-(x'-x)/2)^{2}+(\tilde{f}((x-x')/2)+\tilde{f}((x'-x)/2))^{2}\\
%&\leq & \max_{(x,x')\in [-1,1]^2} (x-x')^{2}+(\tilde{f}(x)+\tilde{f}(x'))^{2}=D^2,
%\end{eqnarray*}
%by using that $\tilde{f}$ is even. We have thus obtained a contradiction, which proves that necessarily, $x'=-x$.

It follows that $K_{\tilde f}$ solves Problem~\eqref{SOP1}. Furthermore, using that $\tilde f$ is even and that each point $(x,\tilde f(x))$ is diametral, associated to $(-x,-\tilde f(x))$, we finally infer that $x^{2}+\tilde{f}(x)^{2}=D^{2}/4$ for all $x\in [-1,1]$. Noting that 
$$
\tilde{f}(x)=\frac{1}{4}(f(x)+g(x)+f(-x)+g(-x)),
$$
every solution $K^\star$ is of the form \eqref{Kstarparam:fg} satisfies \eqref{prop:strip}.
Proposition~\ref{prop:case1} thus follows.
\end{proof}

It remains to investigate the uniqueness of the optimal set, which is the 
purpose of the next result.

\begin{lemma}\label{lem:unique}
Let $K^\star$ be a solution of Problem~\eqref{SOP1}. Then, $K^\star$ is of the form \eqref{Kstarparam:fg}, and for every parametrization $(f,g)$, there exists $\varepsilon>0$ such that:
$$
f(x)=y_D(x)+\varepsilon, \qquad g(x)= y_D(x)-\varepsilon, \quad x\in [-1,1].
$$
\end{lemma}
\begin{proof}
Let $(f,g)$ be a pair of concave positive functions solving Problem~\eqref{SOP1}. In particular, $(f,g)$ satisfies \eqref{prop:strip}. It follows from the proof of Lemma~\ref{lem3} that there exists a continuous odd function $\varphi_o$ on $[-1,1]$ such that
$$
\frac{f(x)+g(x)}{2}=y_D(x)+\varphi_o(x).
$$
 Let $K$ be the convex set defined by \eqref{Kstarparam:fg} where $f$ and 
$g$ are both replaced by $(f+g)/2$. Recall that, according to the proof of Lemma~\ref{lem3}, $K$ is also a solution of Problem~\eqref{SOP1}.
Let us focus on the diameter constraint. Since $K$ solves Problem~\eqref{SOP1}, then one has necessarily
\begin{eqnarray*}
D^2&=&\max_{(x,x')\in [-1,1]^2} (x-x')^{2}+\left(y_D(x)+y_D(x')+\varphi_o(x)+\varphi_o(x')\right)^2\\
&\geq & \max_{x\in [-1,1]} (2x)^{2}+\left(y_D(x)+y_D(-x)\right)^2=D^2.
\end{eqnarray*}
In particular, since every point of $\partial K\cap \{(x,y)\in \R^2\mid x\in (-1,1)\}$ is diametral, the function $[-1,1]\ni x'\mapsto (x-x')^{2}+\left(y_D(x)+y_D(x')+\varphi_o(x)+\varphi_o(x')\right)^2$ is maximal at $x'=-x$. Note that the function $y_D+\varphi_o$ is (concave and therefore) differentiable almost everywhere in $(-1,1)$, and therefore so is $\varphi_o$. Let us consider $x\in [-1,1]$ at which $\varphi_o$ is differentiable. One has 
 $$
\left. \frac{d}{dx'}\left((x-x')^{2}+\left(y_D(x)+y_D(x')+\varphi_o(x)+\varphi_o(x')\right)^2\right)\right|_{x'=-x}=0
 $$
which reads
$
-4x+4y_D(x)(-y_D'(x)+\varphi_o'(x))=0,
$
and after calculation, implies that $\varphi_o'(x)=0$. We infer that $\varphi_o'(x)=0$ for a.e. $x\in (-1,1)$. Since $\varphi_o$ is absolutely 
continuous (and even belongs to $W^{1,\infty}(-1,1)$), we infer that $\varphi_o$ is constant on $(-1,1)$, equal to $\varphi_o(0)=0$. It follows that $(f+g)/2=y_D$ and we infer that
$$
f(x) =y_D(x)+\varphi_e(x)\quad \text{and}\quad 
g(x) = y_D(x)-\varphi_e(x),
$$
where $\varphi_e$ denotes a continuous function on $[-1,1]$. One has for every $x\in [-1,1]$,
\begin{eqnarray*}
D^2&=&\max_{(x,x')\in [-1,1]^2} (x-x')^{2}+\left(y_D(x)+y_D(x')+\varphi_e(x)-\varphi_e(x')\right)^2\\
&\geq  &D^2+4y_D(x)\left(\varphi_e(x)-\varphi_e(-x)\right)+\left(\varphi_e(x)-\varphi_e(-x)\right)^2.
\end{eqnarray*}
and therefore, $4y_D(x)\left(\varphi_e(x)-\varphi_e(-x)\right)+\left(\varphi_e(x)-\varphi_e(-x)\right)^2\leq 0$ so that
$$
-4y_D(x)\leq \varphi_e(x)-\varphi_e(-x)\leq 0.
$$
Inverting the roles played by $x$ and $-x$ in this relation yields that $ 
\varphi_e(x)-\varphi_e(-x)=0$ {and $\varphi_e$ is therefore even.}

By using the same reasoning as above, one shows that for almost every $x$ 
in $(-1,1)$, the derivative of the diameter functional vanishes at $x'=-x$, so that one has $\varphi_e'(x)=0$ a.e. $x$ in $(-1,1)$.  Since $\varphi_e$ belongs to $W^{1,\infty}(-1,1)$ and is in particular absolutely continuous, we infer that $\varphi_e$ is constant on $[-1,1]$. The expected conclusion follows noticing that the converse sense is immediate: every pair $(f,g)$ chosen as in the statement of Lemma~\ref{lem:unique} obviously drives to a solution of Problem~\eqref{SOP1}.  
\end{proof}

\begin{remark}[Geometric interpretation of the proof]
The proof of Lemma~\ref{lem3} can be understood geometrically: indeed, from a solution, we performed two Steiner symmetrizations: one along the strip, and the other in an orthogonal direction. {From the standard properties of Steiner symmetrization (we proved some of them for the sake of completeness) and because of the specific choice of the symmetrization axes, the inradius remains unchanged in this particular case, as well as the area, but the diameter decreases.%The area is unchanged, and the diameter does not increase. 
%And because of the specific choice of the symmetrization axes, the inradius remains unchanged
} The difficulty here lies in proving that the diameter is strictly decreasing, whence the uniqueness. 
\end{remark}

\subsection{Second case: $K^*$ is included in a triangle}\label{triangle}

In that case, the incircle is unique (see Fig~\ref{fig:3cotes}). We assume without loss of generality that it is the unit circle. There are exactly three contact lines (see Def.~\ref{def:contact}), forming a triangle called $T(K)$.
%Let us give a definition of a free boundary and a free zone:

\begin{definition}
We will call ``free boundary $\gamma$ of $\partial{K^{\star}}$'' the union of all non flat parts of  $\partial{K^{\star}}$ and ``free zone'' every 
connected component of the free boundary.
%, each one is included in a connected component of $T\backslash D$ where 
$D$ is the full disk. 
\end{definition}

Recall that according to Lemma~\ref{lem2}, there are at most three free zones located between the contact segments.

\vspace{3mm}

A crucial tool for the analysis is the so-called {\it support function} of the convex body $K$ denoted $h_{K}$. Recall that $h_{K}$ is defined for 
every $\theta \in \mathbb{T}$ by
\begin{equation}
h_{K}(\theta)=\sup_{y\in K}  y\cdot u_\theta 
\end{equation}
where $u_\theta=(\cos(\theta),\sin(\theta)$, and $\mathbb{T}$ is the torus $\R/ [0,2\pi)$. 
We will systematically choose the center of the circle as the origin.
%Note that since $\tilde{h}_{K}$ is 1-homogeneous, it is enough to give $\tilde{h}_{K}$ on the circle of radius 1, representing such vectors by an 
angle $\theta$:
%\begin{equation}
%h_{K}(\theta)=\tilde{h}_{K}( (\cos(\theta),\sin(\theta))).
%\end{equation}
The straight line $D_{\theta}$ whose cartesian equation is $x\cos(\theta)+y\sin(\theta)=h_{K}(\theta)$ is precisely the support line of the convex body $K$ in the direction $u_\theta $ (in what follows, we will also name this direction $\theta$ with a slight abuse of language). 

%It follows that $F_{\theta}:= D_{\theta}\cap K$. $F_{\theta}$ is a segment, possibly reduced to a point. Furthermore, $F_{\theta}$ is a  if, and only if $K$ is flat on this portion. In the case where $F_{\theta}$ is reduced to a point, we will denote this point by $M(\theta)$.
{Let us introduce the sets $F_{\theta}:= D_{\theta}\cap K$. Note that $F_\theta$ is either a segment or a single point. In the latter case, 
we will denote this point by $M(\theta)$}.

Let us finally recall some basic facts on the support function. For a complete survey about this notion, we refer for instance to \cite{Schneider}. When there will be no ambiguity, we will sometimes write $h$ instead of $h_K$.

The support function $h$ associated to a convex body $K$ is periodic, belongs to $H^{1}(\mathbb{T})$ and is $C^{1}$ on the strictly convex parts of $K$. Furthermore, the diameter $D(K)$, area $|K|$ and radius of curvature $R_K$ are respectively given in terms of $h$ by

\begin{equation}\label{DiamSupport}
D(K)=\sup_{(0,2\pi)}\left(h(\theta)+h(\theta+\pi)\right),\quad 
\vert K\vert=\frac{1}{2}\int_{(0,2\pi)}(h^{2}-h'^{2}), \quad 
R_K=h+h''
\end{equation}
where $h''$ has to be understood in the sense of distributions.

Let $\mathcal{T}$ be the set of triangles with unit inradius enclosing $K$. In this section, we will investigate the optimization problem
\begin{equation}\label{SOP:triangle}
\sup_{T\in \mathcal{T}}\sup_{\substack{K\in \mathcal{K}_{r,D}^2\\ K\subset T}}\vert K \vert,
\end{equation}
which can be recast in terms of support functions as
\begin{equation}\label{SOPh}\tag{$\mathcal{P}_h$}
\boxed{\sup_{h\in \mathcal{H}}\frac{1}{2}\int_{(0,2\pi)}(h^{2}-h'^{2})}
\end{equation}
with 
$$\mathcal{H}=\{h \in H^{1}(0,2\pi), h+h''\geq 0\text{ in }\mathcal{D}'(\mathbb{T}), \ \exists T\in \mathcal T \mid  1\leq h\leq h_{T}, \sup_{\theta\in \mathbb{T}} h(\theta)+h(\theta+\pi)\leq D\},$$
where $h_{T}$ is its support function of $T$. Note that $h+h''$  is a positive Radon measure. It is essential to ensure that $h$ is the support function of a convex set. The condition $1\leq h\leq h_{T}$ simply means that $K$, whose support function is $h$, contains the disk $B(0,1)$ and is included in the triangle $T$. 	

Before stating the main result of this section, let us introduce another particular smoothed nonagon, denoted $K_C(D)$.

\begin{definition}[The smoothed nonagon $K_C(D)$]\label{defKC}
Let $D\in]2,2\sqrt{3}[$. We denote by $K_C(D)$ the convex set enclosed in 
an isosceles triangle $\Delta_I$ of inradius $1$ and made of segments and 
arcs of circle of diameter $D$ in
the following way (see Fig.~\ref{fig:McoteOpt}): 
the normal angles to the sides of $\Delta_I$ are
$$
 \eta_{1}=-\pi/2, \quad \eta_2=\arcsin (\tau/2-1)\quad \text{and}\quad {\eta_3=\pi-\eta_2}, 
$$
where {$\tau$ is the unique root in $[2,3]$} of the equation
$$
-\tau^{3}+\left(D^{2}/2+5\right)\tau^{2}-\left(2D^{2}+4\right)\tau+D^{2}=0.
$$

Let us introduce the points $A_i$, $B_i$, $i=1,2,3$ and $M_3$ defined through their coordinates by
$$
A_{i}=
\begin{pmatrix}
\cos\eta_i+h_{i}\sin\eta_i\\
\sin\eta_i-h_{i}\cos\eta_i
\end{pmatrix},
\quad B_{i}=
\begin{pmatrix}
\cos\eta_i-h_{i}\sin\eta_i\\
\sin\eta_i+h_{i}\cos\eta_i
\end{pmatrix},\  i=1,2,3, \quad
M_{1}=(1-\tau)\times
\begin{pmatrix}
\cos(\eta_{1})\\
\sin(\eta_{1})
\end{pmatrix}.
$$
with
$h_1=\sqrt{D^2-\tau^2}$ and $h_2=h_3=\frac{h_1}{4}(\tau-2)$.
The set $K_C(D)$ is then obtained as follows:
\begin{itemize}
\item the points $A_1$, $B_1$, $A_2$, $B_2$, $M_1$, $A_3$, $B_3$ belong to its boundary;  
\item $\arc{B_{2}M_{1}}$ (resp. $\arc{M_{1}A_3}$) and $\arc{A_{1}B_{3}}$ (resp. $\arc{B_{1}A_{2}}$) are diametrally opposed arcs of the same circle of diameter $D$. 
\item the boundary contains the segments $[A_{i}B_{i}]$, $i=1,2,3$. Note that the contact point $I_{i}$ with the incircle is precisely the middle of $[A_{i}B_{i}]$, 
\end{itemize}
Moreover, setting
$$
t_{1}=\arcsin\left(\frac{2(\sin\eta_{1}+h_1\cos\eta_{1})-\tau+2}{D}\right)\quad\text{and }\quad
t_2=\arcsin\left(\frac{\tau}{D}\right),
$$
we have the formula
\begin{equation}\label{formulaKC}
\vert K_C(D)\vert=\frac{\tau}{\tau-2}\sqrt{D^{2}-\tau^{2}}+\frac{D^{2}}{2}(t_2-t_1).
\end{equation}

\begin{figure}[H]
\begin{center}
\includegraphics[width=6cm]{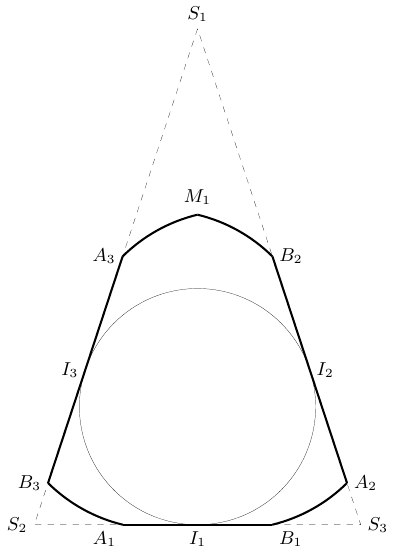}
\caption{The set $K_C(D)$ and its incircle. \label{fig:McoteOpt}}
\end{center}
\end{figure}

% The set $K_C(D)$, made of 4 arcs of circle, pairwise antipodal
\end{definition}

\begin{proposition}\label{prop:KcKe}
Let $D>2$ be given and assume that Problem \eqref{SOP:triangle} has a solution $K^\star$. Then, $K^\star$ is either the set $K_C(D)$ or $K_E(D)$.
\end{proposition}

The end of this section is devoted to proving Proposition~\ref{prop:KcKe}. Hence, let us assume that Problem \eqref{SOP:triangle} has a solution 
denoted $K$ (instead of $K^\star$) for the sake of simplicity. Let $T$ be 
the triangle of inradius 1 containing $K$. Let $D_{\eta_{i}}$ be the three tangent lines to the unit circle defining $T$, where $\eta_{i}$ is the angle between {the horizontal axis} and the normal vector to each side of $T$. We assume that 
 $\eta_{1}<\eta_{2}<\eta_{3}$ and we introduce the contact points $I_{i}$ 
between the line$D_{\eta_{i}}$ and the unit circle. 
%With a slight abuse of notation we will use the same notation for the angle $\eta_{i}$ and $(\cos\eta_i,\sin\eta_i)$ the normal vector of the tangent of the circle at the point $I_{i}$. 
We also define $\varphi_{1}$, $\varphi_{2}$, $\varphi_{3}$ as the demi angles at the center (see Fig.~\ref{fig:triangle}). The problem being rotationally invariant, we will impose without loss of generality that $\eta_{1}=-\pi/2$, and $\varphi_{1}\leq \varphi_{2}\leq \varphi_{3}$. Identifying the index $i$ with the index $i+3$, one has
$$
\varphi_{i}=\frac{\eta_{i+2}-\eta_{i+1}}{2}, \quad i=1,2,3.
$$   The set $K\cap D_{\eta_{i}}$ is a segment (possibly reduced to the point $I_{i}$) denoted $[A_{i},B_{i}]$. The free boundary $\gamma$ being strictly convex according to Lemma~\ref{lem2}, we parametrize it with the help of a function $\theta \mapsto M(\theta)$ defined on $I_{\gamma}$= $(0,2\pi)\backslash\{\eta_{i}\}_{i=1,2,3}$, where $\theta$ is the angle 
between the normal to the support line of the point $M(\theta)$ and the abscissa axis. A point $M$ of the free boundary may have several support lines. More precisely, two cases may arise: either a point has a unique supporting line
 or a point has at least two supporting lines.

Each point $M$ of the second kind is a kind of vertex of $K$ called ``angular point'' of $\partial K$. Moreover, considering the smallest and the largest angle made by its supporting lines, one can associate to $M$ a closed interval $J_{M}\subset I_{\gamma}$.
%In fact, $M$ is a corner point of $\gamma$. 
Notice that two consecutive vertices $M$ and $N$ cannot admit
overlapping intervals $J_{M}$ and $J_{N}$ since it would mean that $\gamma$ contains a  violating
the property that every point in $\gamma$ saturates the diameter constraint. It also implies that angular points of $\gamma$ are isolated, whereas 
points of $\partial K$ of the first kind are represented by a unique angle.

This remark rewrites in the following way in terms of the support function $h$ of $K$:
\begin{itemize}
\item[(i)] if $M(\theta)$ has a unique supporting line, then $\theta+\pi\in I_{\gamma}$ and $h(\theta)+h(\theta+\pi)=D$;
\item[(ii)] in the converse case, there exists $\theta \in J_{M}$ such that $\theta+\pi\in I_{\gamma}$ and $h(\theta)+h(\theta+\pi)=D$.
\end{itemize}

\begin{figure}[H]
\begin{center}
\includegraphics[width=12cm]{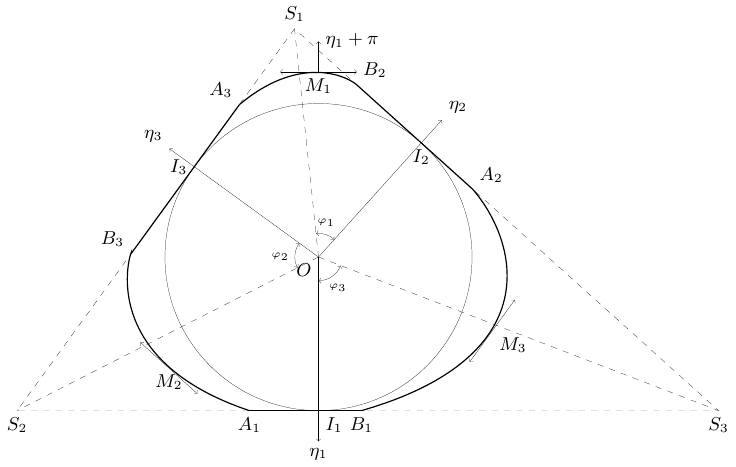}
\caption{Example of a convex $K$ and the triangle $T(K)$.\label{fig:triangle}}
\end{center}
\end{figure}

Regarding the segments $[A_i,B_i]_{i=1,2,3}$, one has
$$
A_{i}=M(\eta_{i}^{-})=\lim_{\theta\to \eta_{i},\theta<\eta_{i}}M(\theta)\quad \text{ and }\quad 
B_{i}=M(\eta_{i}^{+})=\lim_{\theta\to \eta_{i},\theta>\eta_{i}}M(\theta).
$$
For $i=1,2,3$, let $\alpha_{i}$ and $\beta_{i}$ be such that $M(\theta)=A_{i}$ for all $\theta \in [\eta_{i}-\alpha_{i},\eta_{i})$ and $M(\theta)=B_{i}$ for all $\theta \in (\eta_{i},\eta_{i}+\beta_{i}]$. Since angular points are isolated, the free boundary $\gamma$ near $A_{i}$ and $B_{i}$ is made of points of $\partial K$ having a unique supporting line. An easy continuity argument shows that $A_{i}$ and $B_{i}$ saturate the diameter constraint. 
Let us make their diametral point(s) precise. 
Recall that we introduced $F_{\theta}$ as $D_{\theta}\cap K$ and let us characterize $F_{\eta_{i}+\pi}$. Since $\eta_{i+1}-\eta_{i}<\pi$, $\eta_{i}+\pi$ cannot belong to $\{\eta_{j}\}_{j=1,2,3}$, then $F_{\eta_{i}+\pi}$ is a point denoted $M(\eta_{i}+\pi)$ or more simply $M_i$. Considering 
for instance the point $M_{1}$, we have to distinguish between three cases:  
 \begin{itemize}
 \item if $\eta_{1}+\pi \in (\eta_{2}+\beta_{2},\eta_{3}-\alpha_{3})$, meaning that $M_{1}$ lies in the interior of the free boundary, then $M_{1}$ is diametral with both $A_{1}$ and $B_{1}$
 \item if $\eta_{1}+\pi \in (\eta_{2},\eta_{2}+\beta_{2})$, then $M_{1}=B_{2}$ and one easily infers that $M_{1}A_{1}=D$
 \item if $\eta_{1}+\pi \in (\eta_{3}-\alpha_{3},\eta_{3})$, then $M_{1}=A_{3}$ and it follows that $M_{1}B_{1}=D$
 \end{itemize}

\subsubsection{Geometrical description of optimizers}
\begin{lemma}\label{lem5}
Let $i\in \llbracket 1,3\rrbracket$. The contact points $I_{i}$ between the line $D_{\eta_{i}}$ and the incircle is the middle of the  segment $[A_{i},B_{i}]$.
\end{lemma}

\begin{proof}
%{*** A bien v\'erifier ***}
To prove this, we will use a small perturbation of an angle $\eta_{i}$ and get optimality conditions. {Without loss of generality,}
 consider $I_{1}$ and introduce the lengths $l_{A}=I_1A_{1}$ and $l_{B}=I_1B_{1}$. 
{Let us consider the following perturbation: we replace $\eta_{1}$ by $\eta_{1}+\varepsilon$ for $\varepsilon>0$ small, and denote by $T_\eps$ the triangle whose incircle is $B(0,1)$, and whose angles are $\eta_{1}+\eps$, $\eta_2$, and $\eta_3$.} We denote by $L_{\eta_1+\varepsilon}$ the corresponding tangent line of the unit disk.
We now define $J_{\varepsilon}$ as the intersection point  between $D_{\eta_{1}}$ and $
L_{\eta_1+\varepsilon}$. This point satisfies 
$J_\varepsilon=I_1+ \frac{\varepsilon}{2} (-\sin\eta_1,\cos\eta_1)$.
  We build a new convex set included in the triangle $T_{\varepsilon}$ by 
slightly modifying the previous one : {replace $A_1$ and $B_1$ by $A_{\varepsilon}$ and $B_{\varepsilon}$ located on $L_{\eta_1+\varepsilon}$ 
in such a way that the diameter constraint is still fulfilled
(see Fig.~\ref{fig:mil1}).} We explicit the construction of $A_\varepsilon$ below as the intersection of 
$L_{\eta_1+\varepsilon}$ with a well chosen line issued from $A_1$, while 
$B_\varepsilon$ is the intersection of $L_{\eta_1+\varepsilon}$
with the boundary of $K$.
We have to make the balance between 
\begin{itemize}
\item the area we gain: this is triangle $T(A_1J_{\varepsilon}A_\varepsilon)$
\item the area we lose: this is the intersection of $K$ with the half-space $\{ x \cdot u_{\eta_{1}+\varepsilon}\geq 1\}$. At first order, this area is the same than the area of the 
triangle $T(B_1J_{\varepsilon}B_\varepsilon)$
\end{itemize} 

\begin{figure}[H]
\begin{center}
\includegraphics[width=8cm]{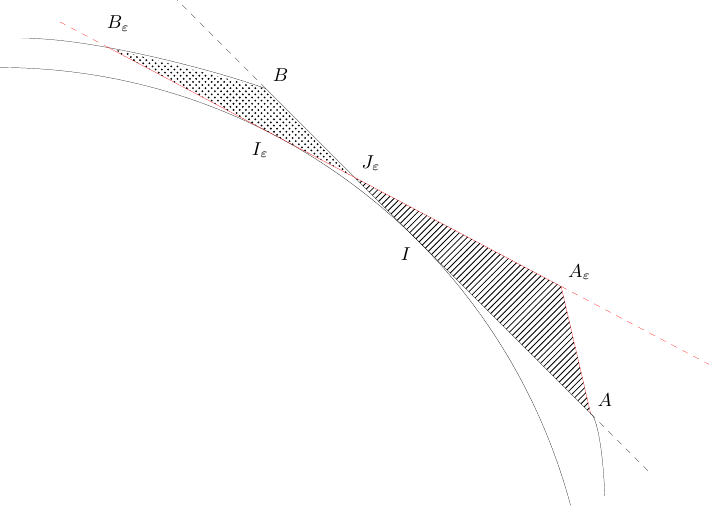}
\caption{Gain of area (strips) vs loss of area (dots)}\label{fig:mil1}
\end{center}
\end{figure}

The two triangles share the same angle $\varepsilon$, therefore the balance of area is
$$\delta A:=\frac{1}{2} \sin\varepsilon \left(J_\varepsilon A_1.J_\varepsilon A_\varepsilon
- J_\varepsilon B_1.J_\varepsilon B_\varepsilon\right) $$
Now we can explicitly compute these lengths and get the expansions
$$
J_\varepsilon A_1=l_A +O(\varepsilon),\qquad J_\varepsilon B_1=l_B +O(\varepsilon), $$
Let us introduce the angle $\theta_A^\varepsilon=\widehat{J_\varepsilon 
A_1 A_\varepsilon}$.
Using elementary trigonometry, we can rewrite the length $J_\varepsilon A_\varepsilon$ as
$$J_\varepsilon A_\varepsilon=\frac{A_1J_\varepsilon}{\cos \varepsilon+\sin \varepsilon \cot \theta_A^\varepsilon}=l_A(1-\varepsilon \cot \theta_A^\varepsilon +o(\varepsilon)).$$
 
Now let us prove that we can choose an angle $\theta_A^\varepsilon$ which 
does not go to zero
while keeping the diameter constraint satisfied. Suppose $\eta_1\in [0,\pi/2]$. 
Recall that $A_1$ is represented by an interval of angles $I_{A_1}=[\eta_1-\alpha,\eta_1]$.
Let $D_{A_1}$ be the set of points that are diametrical to $A_1$ and $\Theta_{A_1}\subset I_{A_1}+\pi$ the set of angles representing elements of $D_{A_1}$:  $$\Theta_{A_1}=\{\theta \in I_\gamma, M(\theta)\in D_{A_1} \mbox{ and } h(\theta)+h(\theta+\pi)=D\}\subset[0,2\pi).$$

We claim that there exists $\gamma>0$ such that for all $\theta'\in [\eta_{1}+\pi-\gamma,\eta_{1}+\pi]$, $\theta'\notin   \Theta_{A}$. Otherwise the diameter constraint on $I_{1}$ would be broken. Let $\zeta=\max(\Theta_{A_1})<\pi+\eta_1$. Choosing $\theta_A^\varepsilon= (\pi+\eta_1-\zeta)/2$ fulfills the desired condition for $\varepsilon$ small enough and provides
a gain of area as $l_A^2 \varepsilon/2+o(\varepsilon))$.

%Since $A$ is the extremity of a flat portion of $K^\star$, points in the 
neighborhood of $A$ on this flat portion does not saturate the diameter constraint. 
%Hence the set of admissible points in this neighborhood of $A$ is made of the intersections of arc of circles of radius $D$ with center in $D_A$. 
one can prove that at the first order in $\varepsilon$, it consists in taking the line with vector $\zeta+\pi$. Take $A_{\varepsilon}$ as the intersection of this line with the tangent of the unit circle with angle $\eta_{1}+\varepsilon$. The desired angle is $\theta=\eta_{1}-\zeta-\pi=O(1).$ This construction is such that $K\cup T(A J_{\varepsilon}A_{\varepsilon})$ still fulfills the diameter constraint as well as the convexity constraint. since $\delta-\pi\geq\eta_1-\alpha$

On the side of $B$ there is no problem with the diameter constraint, thus 
we simply 
observe that $J_\varepsilon B_\varepsilon=l_B+O(\varepsilon)$ by construction.
Therefore we get a loss of area as $l_B^2 \varepsilon/2+o(\varepsilon))$.

Thus we infer that the difference of areas is equal to $\delta A= \frac{\varepsilon}{2} (l_{A}^{2}-l_{B}^{2})+o(\varepsilon)$ which has to be non-positive, which leads to $l_{A}\leq l_{B}$ at the optimum. We repeat the argument with $\varepsilon<0$ to get $l_{B}\leq l_{A}$, whence the equality.
%
%Both triangles for which we have to compute this area are similar to a one with a length side equal to $l>0$ and two basis angles respectively equal to $\theta$ and $\varepsilon$. Its area is given by 
%\begin{equation*}
%l^{2}(1-\varepsilon \vert\mathrm{cotan}(\theta)\vert+o(\varepsilon))\sin(\varepsilon)
%\end{equation*}
%
%%\begin{figure}[H]
%%\begin{center}
%%\includegraphics[width=6cm]{Mil2.pdf}
%%\caption{}\label{fig:mil2}
%%\end{center}
%%\end{figure}
%
%It only remains to prove that we can chose $\theta=O(1)$ and $\pi-\theta=O(1)$ so that $\mathrm{cotan}(\theta)=O(1)$ and the first order term is $l^{2}\varepsilon$.
%
%\vspace{3mm}
%

%
%\vspace{3mm}
%

%
%\vspace{3mm}
%
%Finally observe that in the calculus, $l$ stands respectively for $J_{\varepsilon}A=l_{A}+O(\varepsilon)$ and $J_{\varepsilon}B=l_{B}+O(\varepsilon)$. 
\end{proof}

Now we are going to prove that the free boundary is made of arcs of circle of radius $D/2$ by working on the radius of curvature $R$. {It consists of three steps.} We show first that this radius can only take the values $0$, $D/2$ or $D$ on the free boundary. Then we prove that the set $\{R=D\}$ is necessarily of empty interior to finally deduce that the radius of curvature on non angular points can only be $D/2$.

\begin{lemma}\label{lem6}
On the free boundary $\gamma$ of $K$, the radius of curvature is almost everywhere equal to either $0$, $D/2$ or $D$.
\end{lemma}

\begin{proof}
According to the above discussion, we will distinguish between points of the free boundary $\gamma$ having a unique support line, and angular points.
Since angular points are isolated on $\partial K$, it means that points of $\gamma$  having a unique support line define an open subset $\gamma_1$ 
of $\gamma$ or equivalently that their angle parametrization define an open subset $I_1$ of $I_{\gamma}$= $(0,2\pi)\backslash\{\eta_{i}\}_{i=1,2,3}$. 
%We will denote by $I_1$ this union which corresponds geometrically to a part $\gamma_1$ of $\gamma$. 
Any point of the complement set of $\gamma_1$ is an angular point, and therefore its radius of curvature is zero. Thus, it remains to look at points of $\gamma_1$.

Recall that, since $K$ is a convex set, its radius of curvature defines a 
nonnegative Radon measure.
%Let $J_1$ be an open sub-interval of $I_1$. 
For any $\theta\in I_1$ one has
%\begin{equation}\label{lem6-1}
$h(\theta)+h(\theta+\pi)=D$.
%\end{equation}
Differentiating twice this equality and since $R=h+h''$, one gets that $R+\tau_\pi R=D$ in the sense of measures in $I_1$, where $\tau_\pi$ is 
the translation operator given by $\tau_\pi (f)= f(\pi+\cdot)$ for every continuous function $f$. It follows that  $0\leq  R(\theta) \leq D$ for 
a.e. $\theta$ in $\mathbb{T}$ and thus, $R$ is a bounded function, allowing us to write
\begin{equation}\label{lem6-2}
\forall \theta \in I_1,\quad R(\theta)+ R(\theta+\pi)=D.
\end{equation}

Let us now prove that for almost every $\theta\in I_1$, one has $R(\theta)\in \{0,D/2,D\}$.
Let us assume that the set $\omega=\{\theta\in I_1\mid  0<R(\theta)<D\}$ has a positive measure, otherwise it means that $R=0$ or $R=D$ a.e. 
and we are done. Let us first show that $R$ is necessarily constant on $\omega$.
Let us argue by contradiction: assume there exist two subsets $\omega_1$ and $\omega_2$ such that $|\omega_1|=|\omega_2|>0$ and 
\begin{equation}\label{lem6-3}
\int_{\omega_1} R(\theta) \, d\theta >  \int_{\omega_2} R(\theta) \, d\theta.
\end{equation}
Let us consider a regularization $\xi$ of the function $ v$ defined by
$$
v(\theta)=\left\lbrace \begin{array}{llrl}
+1 &\mbox{ if } \theta\in\omega_1, & - 1 &\mbox{ if } \theta\in\omega_1 +\pi \\
-1 &\mbox{ if } \theta\in\omega_2 ,&  1 & \mbox{ if } \theta\in\omega_2 +\pi \\
\end{array}\right.$$
and we will deal with the perturbation $h+\varepsilon v$ of the support function $h$ for $\varepsilon >0$ small.
In what follows, we should deal with the regularization $\xi$, work on a subset of $\omega$ on which $0<\eta\leq h(\theta)$, and finally pass to the limit $\eta \searrow 0$. To avoid technicalities, we will directly write the asymptotic of the derivative of the area under this perturbation, with a slight abuse of notation.

Since the area of the domain is 
 $$\vert K\vert =J(h) \quad \text{where } J(h)=\frac12 \int_0^{2\pi} (h^2(\theta)- {h'}^2(\theta))\, d\theta,$$ the
first derivative of the area under the perturbation above reads as
$$
\langle dJ(h),\xi\rangle= \int_{\omega_1\cup\omega_2\cup (\omega_1+\pi)\cup(\omega_2+\pi)} h\xi-h'\xi' =
\int_{\omega_1\cup\omega_2\cup (\omega_1+\pi)\cup(\omega_2+\pi)} (h+h'')\xi .$$
By definition of $\xi$, one gets
$$
\langle dJ(h),\xi\rangle=
\int_{\omega_1}R  - \int_{\omega_2}R - \int_{\omega_1+\pi}R + \int_{\omega_2+\pi}R
$$
and according to \eqref{lem6-2}, it comes
$$
\langle dJ(h),\xi\rangle=
 \int_{\omega_1}R - \int_{\omega_2}R  - \int_{\omega_1} (D-R)
+ \int_{\omega_2} (D-R) =2 \left(\int_{\omega_1}R  - \int_{\omega_2}R\right) >0
$$
leading to a contradiction. It follows that $R$ is necessarily constant on $\omega$. Let us moreover show that the constant value of $R$ is precisely $D/2$. We proceed similarly: let us choose a perturbation $\xi$ equal 
to $1$ on a subset $\omega_1$ and $-1$ on $\omega_1+\pi$. The same computation as above leads to 
$$
\langle dJ(h),\xi\rangle= \int_{\omega_1}R  - \int_{\omega_1} (D-R) =\int_{\omega_1}(2R-D) ,
$$
and we conclude since this derivative must be zero (indeed, if this derivative would not vanish, either the admissible perturbation $\xi$ or $-\xi$ would make the area increase). We conclude that necessarily $R\in \{0,D/2,D\}$ on $I_{1}$.

%To go further, we need to prove that the sets $\{R=0\}$,$\{R=D/2\}$, 
$\{R=D\}$ are unions of intervals and to locate them. For that purpose we will now study the perturbation on $R=h+h''$. By definition of $R$, it is a radon measure such that $\int_{(0,2\pi)}\cos dR=\int_{(0,2\pi)}\sin dR=0$ . Now Suppose that $R$ is optimal. Let $J\subset I_{1}$ and consider a perturbation $\xi$ on $J\cup J+\pi$ such that 
%$$
% \left\{
%    \begin{array}{ll}
%        \xi\leq 0 & \mbox{on } S_{D}\\
%         \xi\geq 0 & \mbox{on } S_{0}\\
%         \xi(\theta+\pi)=-\xi(\theta) & \forall \theta\in J
%               
%    \end{array}
%\right.
%$$
%
%The derivative of the area under any perturbation $R+t\xi$ is equal to
%
%$$\frac{dA}{dt}\setminus_{t=0}=\int_{(0,2\pi)} h\xi d\theta.$$
%
%After computations it comes:
%
%$$ \frac{dA}{dt}\setminus_{t=0}=\int_{J} (2h-D)\xi d\theta.$$
%
%
%Optimality conditions yield the existence of Lagrange multipliers $\lambda,\mu$ such that 
%
%
%\begin{equation}\label{lem6-4}
%\psi= h-D/2-\lambda\cos-\mu\sin
% \left\{
%    \begin{array}{ll}
%       \geq 0 & \mbox{on } S_{D}\\
%         \leq 0 & \mbox{on } S_{0}\\
%         =0 & \mbox{on } S_{D/2}
%               
%    \end{array}
%\right.
%\end{equation}
%
%
%
%
%
%Let $J$ be a connex component of $I_{1}$, then $J+\pi\subset I_{1}$ and for all $\theta\in J$, $h(\theta)+h(\theta+\pi)=D$. Furthermore, $h$ is 
$C^{1}$ and $h''\in L^{\infty}$ with $h+h''=R$. Differentiating twice and adding yields $\psi+\psi''=R-D/2$. suppose for example that  $J_{D}=J\cap S_{D}$ has nonempty interior. Then $\psi\geq 0$ on $J_{D}$ and $\psi+\psi''=D$. Let $\psi$ evolve freely along the differential equation until $\psi$ vanishes and  goes below $0$. Then $psi$ becomes negative and 
is ruled by the differential equation $\psi+\psi''=0$

\end{proof}

From this lemma we deduce that if the boundary $\partial K$ contains an 
arc of circle of radius $D/2$, it also contains its antipodal part (in other words the set of points of $\partial K$ diametrically opposed to those of the arc of circle), and if it contains an arc of circle of radius $D$, it also contains its center. Let us show that this second case cannot occur, following an idea in \cite{BelOud}.

\begin{lemma}\label{lem7}
The two assertions are incompatible:
\begin{itemize}
\item the free boundary $\gamma$ contains an arc of circle of radius $D$;
\item its center belongs to $\partial K$.
\end{itemize}
\end{lemma}

\begin{proof}
Let us argue by contradiction. Let us denote by $C$ the circle of radius $D$ one arc of which belongs to $\gamma$ and by $P\in \partial K$ its center. 
Note that since $C$ saturates the diameter constraint, according to lemma 
\ref{lem1}, it belongs to the free boundary $\gamma$ or lies in the intersection of two segments. In this last case $K$ has only two free zones and C is an edge of $T$. Anyway $C$ is not in the neighborhood of any contact point. By choosing adequately an orthonormal basis, assume that the coordinates of $P$ are $(-D/2,0)$ and the coordinate of the center of the arc, denoted by $Q$, are $(D/2,0)$. Now for $\varepsilon>0$ consider $Q_{\varepsilon}$ whose coordinates are $(D/2+\varepsilon,0)$ and define 
$$K_{\varepsilon}=\operatorname{hull}(K\cup Q_{\varepsilon})\cap B(Q_{\varepsilon},D).$$
where $B(Q_{\varepsilon},D)$ is the disc of center $Q_\varepsilon$ and radius $D$.
\begin{figure}[H]
\begin{center}
\includegraphics[width=6cm]{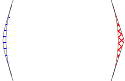}\hspace{4cm}
\includegraphics[scale=0.77]{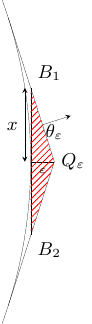}
\caption{Left: gain of area (red crosshatch) vs loss of area (blue horizontal lines). Right: calculus of the gain. \label{fig:ArcDimp}}
\end{center}
\end{figure}

Since the free boundary is modified locally, far from the contact point, the inradius remains unchanged and the diameter also by construction. This transformation drives to a gain of area on the right part, and a loss on the left part (see Fig.~\ref{fig:ArcDimp}). Let us show that the gain is $O(\varepsilon\sqrt{\varepsilon})$ and the loss is $O(\varepsilon^2)$.	
\begin{itemize}
\item {\bf gain:}
%\begin{figure}[H]
%\begin{center}
%\caption{Calculus of the gain \label{fig:ArcDimp2}}
%\end{center}
%\end{figure}
using the notations on the right part of Fig.~\ref{fig:ArcDimp}, one determine a lower bound of the area gain by computing the area of the triangle  $B_{1}Q_{\varepsilon}B_{2}$. Here $x=\varepsilon/\tan(\theta_{\varepsilon})$ with $\cos(\theta_{\varepsilon})=D/(D+\varepsilon)$, and therefore, $x=\operatorname{O}(\sqrt{\varepsilon})$, and thus, a lower bound 
on the area gain is $\operatorname{O}(\varepsilon\sqrt{\varepsilon})$.
\item {\bf loss:} note that if the radius of curvature is $D$ on an open interval, thus it is equal to $0$ on its antipodal interval. It means that the center of the corresponding arc of circle is an angular point, and hence it admits two different tangent lines. By convexity, the loss area is less than the one of the triangle formed by the point $P$, and the two 
intersection points of the tangent with the circle $C(Q_{\varepsilon},D)$. Now the angle of the tangents does not depend on $\varepsilon$, and the same kind of calculus shows that the area loss is $\operatorname{O}(\varepsilon^{2})$.
\end{itemize}
Hence, choosing $\varepsilon>0$ small enough guarantees that $|K_\varepsilon|>|K|$ and we have thus reached a contradiction.
\end{proof}

Let us complete the description of the free boundary with the help of two 
lemmas.\\

\begin{lemma}\label{lem8}
The free boundary $\gamma$ of $K$ is the union of arc of circles of diameter $D$ (i.e. the radius of curvature is equal almost everywhere to $D/2$ 
on $\gamma$), that are mutually antipodal.
\end{lemma}

\begin{proof}
%Lemma \ref{lem6} shows that the radius of curvature of an optimal set is 
either $0, D/2$ or $D$ on the free boundary, and lemma \ref{lem7} shows that on every interval $I$ where the relation $h(\theta)+h(\theta+\pi)=D$ holds, the curvature cannot be $0$ or $D$ in any subinterval. Otherwise 
we would have an arc of circle of diameter $D$, which is impossible.
%Now we follow the same scheme as in lemma $\ref{lem6}$, but now with perturbations on the radius of curvature $R$.
As usual, we denote the optimal set by $K$ in this proof.
We will consider its radius of curvature $R$ as a variable. Recall that, globally,  $R$ is a Radon measure on $\mathbb{T}$ such that 
\begin{equation}\label{eq:orthcossin}
\langle R,\cos\rangle_{\mathcal{M}(\mathbb{T}),\mathscr{C}^0(\mathbb{T})}=0=\langle R,\sin\rangle_{\mathcal{M}(\mathbb{T}),\mathscr{C}^0(\mathbb{T})}=0
\end{equation}
(we choose here to fix the origin at the Steiner point of the convex set $K$). Its associated support function  $h$ solves the ODE
\begin{equation}\label{lem8:1}
\left\{\begin{array}{ll}
h+h''=R & \text{in }\mathbb{T}\\
\int_0^{2\pi}h(\theta)e^{i\theta}\, d\theta=0 & 
\end{array}
\right.
\end{equation}
Let $F$ be the associated resolvent operator, in other words,
$$
F:\mathcal{R}_D\ni R\mapsto F[R]=h\in H^1(\mathbb{T}),
$$
where $h$ is the unique solution to System~\eqref{lem8:1} and
$$
\mathcal{R}_D=\left\{R\in \mathcal{M}(\mathbb{T}) \mid \langle R,\cos+i\sin \rangle_{\mathcal{M}(\mathbb{T}),\mathscr{C}^0(\mathbb{T})}=0\text{ and } F[R](\theta)+F[R](\theta+\pi)\leq D, \  \theta \in \mathbb{T} \right\}.
$$
In what follows and for the sake of notational simplicity, we will denote 
the quantity $\langle R,f\rangle_{\mathcal{M}(\mathbb{T}),\mathscr{C}^0(\mathbb{T})}$, where $f$ is a continuous function in $\mathbb{T},$ by $\int_0^{2\pi}R(\theta) f(\theta)\, d\theta$  with a slight abuse.

We recall that the area of $K$ is given by
\begin{equation}\label{lem9:2}
\vert K\vert =J(R) \quad \text{where }J(R)=\int_{0}^{2\pi} F[R](\theta)R(\theta)\, d\theta.
\end{equation}
%We introduce the set of admissible radii of curvature, given by 
%$$
%\mathcal{A}=\left\{R\in \mathcal{M}_+(\mathbb{T})\mid \eqref{eq:orthcossin}\text{ is satisfied and } F(R)(\theta)+F(R)(\theta+\pi)\leq D, \  \theta \in \mathbb{T} \right\},
%$$ 
%where $\mathcal{M}_+(\mathbb{T})$ is the set of positive Radon measures on $\mathbb{T}$. 

Let $R$ be the radius of curvature function of the optimal set $K$, and $h=F(R)$. Let $I$ denote a subset of $(0,\pi)$ of positive measure (assumed to contain an interval without loss of generality since angular points are isolated) on which there holds $h(\theta)+h(\theta+\pi)=D$. 
According to Lemma~\ref{lem6}, $R$ is bounded on $I$, such that $R(\theta)+R(\theta+\pi)=D$ and $R\in \{0,D/2,D\}$ a.e. on $I$. 
%Let us introduce the notation $I_{\alpha}=I\cap\{R=\alpha\}$, and assume by contradiction that $\vert I_{0}\vert > 0$ or $\vert I_{D}\vert> 0$. 
Moreover, according to Lemma~\ref{lem7}, the interiors of $I\cap \{R=0\}$ and $I\cap \{R=D\}$ are empty. 

We want to write the optimality conditions satisfied by $R$ locally on the interval $I$.
For that purpose we need to use admissible deformations: these are precisely deformations
$\xi$ belonging to the tangent cone at $R$, we recall this definition:
{\it the tangent cone} to the set $L^\infty(I;[0,D])$ at $R$, (also called the \textit{admissible cone}) denoted $\mathcal{T}_{R}$ is the set of functions 
$\xi\in L^\infty(I)$ such that, for any sequence of positive real numbers 
$(\eta_n)_{n\in \N}$ decreasing to $0$, there exists a sequence of functions $\xi_n\in L^\infty(I)$ converging to 
$\xi$ as $n\rightarrow +\infty$, and $R+\eta_n\xi_n\in L^\infty(I;[0,D])$ 
for every $n\in \N$.

Let us now give the first order optimality condition. This is a quite classical result in control
theory, but for sake of completeness, we postpone the proof of the following Lemma to Appendix~\ref{sec:prooflemopcond}.
\begin{lemma}\label{lemopcond}
There exist three real numbers $(\mu,\alpha,\beta)$ (Lagrange multipliers), which are not
all zero, such that
the radius of curvature $R$ of the optimal domain and its support function $h$ satisfy
\begin{equation}
\forall \xi\in \mathcal{T}_{R}, \quad \int_I\left(\mu(2h(\theta)-D)+ \alpha \cos \theta+ \beta \sin \theta\right)\xi (\theta)\, d\theta\leq 0.
\end{equation}
\end{lemma}
To finish the proof of Lemma \ref{lem8}, let us introduce the switching function 
$$
\Psi_{R}:\theta\mapsto \mu(2h(\theta)-D)+ \alpha \cos \theta+ \beta \sin \theta,
$$ 
where $h$ is the solution to \eqref{lem8:1} associated to $R$. The first order necessary condition can be recast as
$$
\forall \xi \in \mathcal{T}_{R}, \quad \int_I \Psi_{R}\xi\leq 0.
$$
Let $y_0\in I$ be a Lebesgue point of $I\cap \{R=0\}$ and let $(G_{n})_{n\in \N}$ denote a subset of $I\cap \{u^\star=0\}$ containing $y_0$. Then, $\xi=\mathds{1}_{G_n}$ belongs to $\mathcal{T}_{R}$ and therefore
$$
\int_{G_{n}} \Psi_{R}\leq 0.
$$
By dividing this inequality by $|G_{n}|$ and letting $G_{n}$ shrink to $y_0$ as $n\to +\infty$, we infer that $\Psi_{R}(y_0)\leq 0$ according to the Lebesgue density theorem.

Generalizing this reasoning to the sets $I\cap \{R=D\}$ and $I\cap \{0<R<D\}$, it follows that
\begin{itemize}
\item on $I\cap \{R=0\}$, $\Psi_{R}\leq 0$;
\item on $I\cap \{R=D\}$, $\Psi_{R}\geq 0$;
\item on $I\cap \{0<R<D\}$, $\Psi_{R}= 0$.
\end{itemize}
Note that $\Psi_{R}$ is continuous. Let us distinguish between two cases. 
If $\mu=0$, then $\Psi_{R}(\theta)= \alpha \cos \theta+ \beta \sin \theta$ with $(\alpha,\beta)\neq (0,0)$ and then, $\{\Psi_{R}=0\}$ has zero measure. It follows that $R$ is bang-bang, equal to 0 and $D$ almost everywhere in $I$. By continuity, since $I$ contains an interval, one has either $R=0$ or $R=D$ on an interval, which is in contradiction with Lemma~\ref{lem7}. 
In the same way, if $\psi_R<0$ (or $\psi_R>0$) somewhere, it will remain negative (or positive)
on an interval, implying that $R=0$
on that interval, in contradiction with Lemma~\ref{lem7}. Therefore, we deduce that $\psi_R$
is identically zero which implies that
$$
h=\frac{D}{2}+\frac{\alpha}{\mu} \cos+\frac{\beta}{\mu} \sin \quad \mbox{on } I.
$$
The same identities hold true on $I+\pi$, which corresponds to an antipodal arc of circle. The expected result follows. Notice finally that, since 
angular points are isolated (which allowed us to assume that $I$ contained an open interval), $\gamma$ is the union of arcs of circle of diameter $D$. 
\end{proof}

Another necessary point is to determine when ones switches from an arc of 
circle to another one.
\begin{lemma}\label{lem9}
Arc of circles only end at an angular point of the free boundary. Furthermore, the only angular points in the interior of the free boundary are the points $M_{i}$, $i=1,2,3$.
\end{lemma}

\begin{proof}
We have seen that a piece of $\gamma$ whose points have a unique supporting line corresponds to an arc of a given circle with diameter $D$. All such points are represented by a unique angle. Hence, denoting by $I$ the corresponding interval of angles, the relation $h(\cdot)+h(\cdot+\pi)=D$ 
holds true on $I$.
It follows that an arc of circle breaks in the interior of $\gamma$ if, and only if there exists an angular point $M$ represented by an interval $J_{M}$ on which the relation $h(\cdot)+h(\cdot+\pi)=D$ is not satisfied 
(otherwise we would necessarily have $R=D$ on $J_M$ because of Lemma~\ref{lem6}, which is impossible because of Lemma~\ref{lem8}). Therefore, only an angular point can break an arc of circle and we claim that such a point is necessarily one of the points $M_{1}$, $M_2$, $M_3$. Indeed, let us write $J_{M}=[\alpha,\beta]$ with $\alpha \leq \beta$ and recall that for $\varepsilon>0$ small enough, $\theta\in[\alpha-\varepsilon,\alpha]$ (and respectively $\theta\in[\beta,\beta+\varepsilon]$) is associated to a point on an arc of circle with diameter $D$. Let $A$ (resp. $B$) be the points of $\partial K^\star$ corresponding  by $\alpha+\pi$ (resp. $\beta+\pi$). If $A=B$, there are two pairs of arc of circle with same center, same radius meeting with a nonzero angle, which is impossible. Thus, 
one has $A\neq B$ and there is a point in the boundary between $A$ and $B$ which does not saturate the diameter constraint (otherwise, using the same arguments as above, there would exist an arc of circle of radius $D$ between $A$ and $B$). This point belongs necessarily to a contact line, which proves that $J_{M}$ contains one of the angles $\eta_{i}+\pi$, $i=1,2,3$. It follows that $M$ corresponds to a point $M_{i}$, $i=1,2,3$.
\end{proof}

According to Lemma~\ref{lem8} and Lemma \ref{lem9}, each free zone of $\gamma$ is made of one or two arc of circles, and for each one, the antipodal arc of circle is in $\gamma$.

We end our study by distinguishing between two cases, depending on whether $\gamma$ is made of two or three free zones.

%\subsubsection{Case of one free zone}
%
%In this case, the only free zone is necessarily diametrical to edges of the triangle. So it is made at least of one arc of circle of radius $D$ with an edge as center. This is impossible by lemma \ref{lem7}.

\subsubsection{Case of two free zones}
First of all, let us remark that the case where the boundary contains only one free zone
cannot occur. Indeed, it would mean that all the points in this free zone, that we know
to be diametral, would be at the distance $D$ of one vertex of the triangle. But this
is impossible, according to Lemma~\ref{lem7}. Thus, it remains to look at 
the case of
two free zones. In that case, one of the vertices of the triangle belongs 
to the boundary $\partial K^\star$. Exactly for the same reason, it is impossible that
one piece of the free boundary is diametral to this vertex.
Therefore,  
the two remaining free zones that we denote $ Z_{1}$ and $Z_{2}$ are mutually diametral, which means that for each $M_{1}$ in $Z_{1}$ there exists 
$M_{2}$ in $Z_{2}$ with $M_{1}M_{2}=D$.

\begin{figure}[h!]
\begin{center}
\includegraphics[width=6cm]{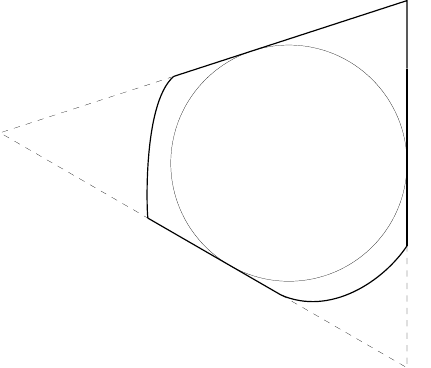}
\caption{A convex set with two free zones \label{fig:2ZL}}
\end{center}
\end{figure}
The case of two free zones arises whenever some points $A_i$ and $B_i$ on 
Fig.~\ref{fig:triangle}  coincide  with a vertex $S_i$. According to Lemma~\ref{lem5}, the contact point are the middle of the contact segments. Moreover, two segments have a vertex as endpoint, and it is necessary for the contact segment to be included in the edges of the triangle that this 
vertex is closer to the contact points than the other vertices.
With the notations previously introduced  (and summed-up on Fig.~\ref{fig:triangle}), we have $S_{i}I_{j}=\tan\varphi_{i}$ for $i\neq j$. 
Since we assumed that  $0<\varphi_{1}\leq \varphi_{2}\leq \varphi_{3}<\pi/2$, the vertex is necessarily $S_{1}$ and one has $I_{2}A_{2}=I_{3}B_{3}=\tan\varphi_{1}$.

{Assume hence without loss of generality that $Z_1$ contains $A_{1}$. Since $A_{1}$ is
diametral, there exists $M\in Z_{2}$ such that $MA_1=D$. We are going to prove that $M$ is unique and equal to $A_2$.
Assume by contradiction that it is not the case. Then there exists an angle $\theta \notin [\eta_2-\alpha_2,\eta_2]$ representing $M$ with $\theta+\pi\in[\eta_1-\alpha_1,\eta_1]$ and $h(\theta)+h(\theta+\pi)=D$. Consider $\varepsilon>0$ small such that $\eta_2 -\alpha_2-\varepsilon>\theta$. Since the only angular point is a point $M_i$, every angle $\theta'\in]\eta_2-\alpha_2-\varepsilon,\eta_2-\alpha_2[$ uniquely represents a point that is diametral. We deduce that for all $\theta'\in ]\eta_2-\alpha_2-\varepsilon,\eta_2-\alpha_2[$, $h(\theta')+h(\theta'+\pi)=D$. From the 
inequalities: $\theta+\pi \geq \eta_1-\alpha_1$ and $\theta'>\theta$ we obtain that $\theta'+\pi \geq \eta_1-\alpha_1$. The inequality $\eta_2-\eta_1<\pi$ guarantees that $\theta'+\pi \in [\eta_1-\alpha_1,\eta_1]$, which means that every point represented by the angles $\theta'\in  ]\eta_2-\alpha_2-\varepsilon,\eta_2-\alpha_2[$ are diametral to $A_1$, hence the existence of an arc of radius $D$, which is impossible.
}

Assume hence without loss of generality that $Z_1$ contains $A_{1}$. Since $A_{1}$ is
diametral, there exists $M\in Z_{2}$ such that $MA_1=D$. Assume by contradiction that $A_2$ is not diametral to $A_1$, hence there is a unique supporting line at $M$. Let $\theta$ be the angle associated to this support line. By uniqueness of the supporting line, one has necessarily $h(\theta)+h(\theta+\pi)=D$ with $\theta+\pi\in [\eta_{1}-\alpha_{1},\eta_{1}]$. Then, every point "above" $M$ is represented by a unique angle $\theta'>\theta$ and we have $h(\theta')+h(\theta'+\pi)=D$ but $\theta'+\pi>\theta+\pi$, so the angle $\theta'+\pi$ also represents $A_{1}$.
 It shows that every point above $M$ is diametral to $A_{1}$. In particular, $A_{1}$ and $A_{2}$ are diametral, whence the contradiction. Similarly, one shows that $B_{3}B_{1}=D$. 
 
 Recall that the free zones are only made of arc of circles of diameter $D$. Let us show that each free zone is one arc of circle, that is antipodal to the other free zone. 
 If it were not the case, one point $M_{i}$ with $i=2,3$ would be in the interior of the free zone. Let us consider without loss of generality that $M_3$ belongs to the interior of the free boundary.  
 %Hence, according to Lemma~\ref{lem9}, it is necessarily diametral with $S_1$ or $B_3$. 
 Let $N$ be a point of $\gamma$ strictly between $B_1$ and $M_3$. Let $\theta$ be the corresponding angle of the associated supporting line, which 
is unique. Then, $\theta<\eta_3+\pi$ and $N$ is diametral with a point whose angles set of its supporting line(s) is included in $(\eta_2,\eta_3)$. It is necessarily $S_1$. But this is impossible according to Lemma~\ref{lem7} since $\gamma$ cannot contain an arc of circle of radius $D$ whose center is a vertex of $T$.  

Therefore, the free zones are antipodal arcs of circle of radius $D/2$. 
Since the points $A_{1}$, $B_{1}$, $A_{2}$, $B_{3}$ belong to the same circle and are two by two diametral, they are the vertices of a rectangle, meaning that $T$ is an isosceles triangle
(we use here the fact that the incircle and the rectangle share the same axis of symmetry). Taking the convention that $\eta_{1}=-\pi/2$, we have $\eta_{3}=\pi-\eta_{2}$ and $\varphi_{1}=\pi/2-\eta_{2}$ (see Fig~\ref{fig:2ZLopt}). 

\begin{figure}[H]
\begin{center}
\includegraphics[width=12cm]{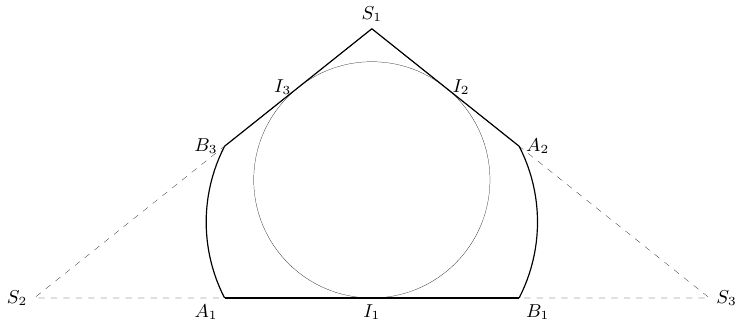}
\caption{Picture of an admissible set with two free zones \label{fig:2ZLopt}}
\end{center}
\end{figure}

Now let us compute the exact value of $\eta_{2}$ with respect to $D$. Since $\varphi_{1}\leq \pi/3$, one has necessarily $\eta_{2}\geq \pi/6$. 

Let us consider the orthonormal basis $(O;\frac{\overrightarrow{A_1B_1}}{A_1B_1},\frac{\overrightarrow{A_1B_3}}{A_1B_3})$ centered at $O$, the incircle center. Since the abscissa of $A_1$ is the same as the one of $B_3$ 
and since $I_3$ is the middle of $[S_1B_3]$ (and resp. $I_2$ is the middle of $[S_1A_2]$), we infer that the coordinates of $A_1$ and $A_2$ are then  
$$
A_{1}=(-2\cos\eta_{2},-1)\quad \text{and}\quad A_{2}=\left(2\cos\eta_{2},\frac{\cos(2\eta_{2})}{\sin\eta_{2}}\right).
$$

Solving the equation $A_{1}A_{2}=D$ leads to the polynomial equation:

\begin{equation}\label{2ZLeq1}
P(\sin\eta_2)=0\quad \text{with }P(X)=X^{3}-\frac{D^{2}-1}{4}X^{2}-\frac{1}{2}X+\frac14 .
\end{equation}
We need to determine a solution in $[1/2,1]$. Assume that $D>2$. 
Let us observe that $P(1)=\frac{4-D^{2}}{4} <0$ and $P(1/2)=\frac{3-D^{2}}{16}<0$. Furthermore, one shows easily that $P$ is either decreasing 
on $(1/2,1)$ or decreasing and then increasing on $(1/2,1)$. Thus the equation $P(\sin\eta_2)=0$ has no solution on $[1/2,1].$
 We conclude that this is not possible to build an optimal set with two free zones.

\subsubsection{Case of three free zones}
Let us distinguish between two cases.
\begin{center}
{\bf Subcase 1: all the points $M_{i}$, $i=1,2,3$ belong to the interior of $\gamma$.} 
\end{center}
In this case, the previous study has shown that the free boundary is as follows (see Fig.~\ref{fig:triangle3ZL})
\begin{itemize}
\item
$\arc{A_{3}M_{1}}$ and $\arc{B_{1}M_{3}}$ are antipodal arcs of circle of 
radius $D/2$,
\item
$\arc{A_{2}M_{3}}$ and $\arc{B_{3}M_{2}}$ are antipodal arcs of circle of 
radius $D/2$,
\item
$\arc{A_{1}M_{2}}$ and $\arc{B_{2}M_{1}}$ are antipodal arcs of circle of 
radius $D/2$,
\item $I_{i}$ is on the middle of $[A_{i},B_{i}]$
\item $M_{i}$ is on the perpendicular bisector of $[A_{i},B_{i}]$ (or $I_{i},O$ and $M_{i}$ are aligned).

\end{itemize}

We deduce the relationships
\begin{equation}\label{3ZL:1}
%\left\{
%    \begin{array}{l}
       \overrightarrow{M_{3}B_{1}}=\overrightarrow{M_1A_3}, \quad 
       \overrightarrow{M_1B_2}=\overrightarrow{M_2A_1}, \quad \text{and}\quad 
       \overrightarrow{M_2B_3}=\overrightarrow{M_3A_2}.
%    \end{array}
%\right.
\end{equation}

\begin{figure}[H]
\begin{center}
\includegraphics[width=10cm]{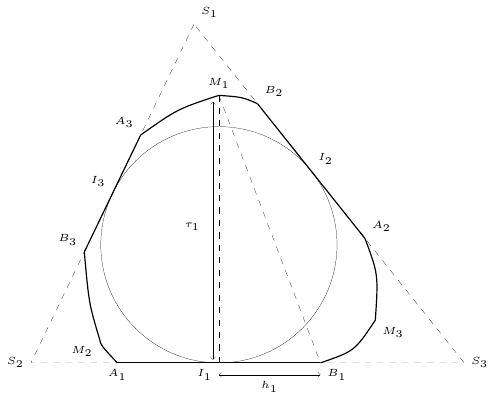}
\caption{Case of three free zones and the $M_{i}$'s belong to the interior of the free zones. \label{fig:triangle3ZL}}
\end{center}
\end{figure}

Let $\tau_{i}=M_{i}I_{i}$ and $h_{i}=I_{i}A_{i}$.
Then necessarily $\tau_{i}>2$ and we have the relationship
\begin{equation}\label{3ZL:2}
h_{i}=\sqrt{D^{2}-\tau_{i}^{2}}
\end{equation}
Let us consider the orthonormal basis $(O;\frac{\overrightarrow{I_1B_1}}{I_1B_1},\frac{\overrightarrow{I_1O}}{I_1O})$ centered at $O$, the incircle center. For $i=1,2,3$, the coordinates of $A_i$, $B_i$ and $M_i$ are
$$
A_{i}=
\begin{pmatrix}
\cos\eta_i+h_{i}\sin\eta_i\\
\sin\eta_i-h_{i}\cos\eta_i
\end{pmatrix}
, \quad 
B_{i}=
\begin{pmatrix}
\cos\eta_i-h_{i}\sin\eta_i\\
\sin\eta_i+h_{i}\cos\eta_i
\end{pmatrix},\quad
M_{i}=(1-\tau_{i})\times
\begin{pmatrix}
\cos\eta_i\\
\sin\eta_i
\end{pmatrix}.
$$
By assimilating the index $i$ with $i+3$, the vector relationships above rewrites
\begin{equation}\label{3ZL:3}
i=1,2,3,\quad \left\{
	  \begin{array}{ll}
2-\tau_{i}&=(2-\tau_{i+1})\cos(\eta_{i}-\eta_{i+1})+h_{i+1}\sin(\eta_{i}-\eta_{i+1})\\
h_{i}&=(2-\tau_{i+1})\sin(\eta_{i}-\eta_{i+1})-h_{i+1}\cos(\eta_{i}-\eta_{i+1}),
\end{array}
\right.
\end{equation}
from which we infer that
\begin{equation}\label{3ZL:4}
\left\{
	  \begin{array}{ll}
(2-\tau_{1})\tan(\eta_{3}-\eta_{2})&=h_{1}\\
(2-\tau_{2})\tan(\eta_{1}-\eta_{3})&=h_{2}\\
(2-\tau_{3})\tan(\eta_{2}-\eta_{1})&=h_{3}.
\end{array}
\right.
\end{equation}
%
%For which we get:
%
%\begin{equation}\label{3ZL:5}
%\left\{
%	  \begin{array}{ll}
%(2-\tau_{1})\tan(\eta_{3}-\eta_{2})&=h_{1}\\
%(2-\tau_{2})\tan(\eta_{1}-\eta_{3})&=h_{2}\\
%(2-\tau_{3})\tan(\eta_{2}-\eta_{1})&=h_{3}
%\end{array}
%\right.
%\end{equation}
With the value of $h_{i}$ given by \eqref{3ZL:2}, we have the quadratic equation on $\tau_{1}$:
\begin{equation}\label{3ZL:6}
(2-\tau_{1})^{2}\tan^{2}(\eta_{3}-\eta_{2})=D^{2}-\tau_{1}^{2}.
\end{equation}
and similarly for the others. This yields
\begin{equation}\label{32L:7}
2-\tau_{1}=2\cos^{2}(\eta_{3}-\eta_{2})\pm\cos(\eta_{3}-\eta_{2})\sqrt{D^{2}-4\sin^{2}(\eta_{3}-\eta_{2})}.
\end{equation}

Since $2-\tau_{1}$ is negative, we can choose the sign depending on the value of $\cos$. Recall that $\eta_{i+1}-\eta_{i}\in (0,\pi)$ and $\eta_{3}-\eta_{2}\leq \eta_{1}-\eta_{3}\leq \eta_{2}-\eta_{1}$. Furthermore, $\eta_{i+1}-\eta_{i}\in (0,\pi/2)$ means that the triangle has an obtuse angle. This can happen only once, and for $\eta_{3}-\eta_{2}$. So at least $\eta_{1}-\eta_{3}$ and  $\eta_{2}-\eta_{1}$ are in $(\pi/2,\pi)$ and their cosine is negative. Assuming now that we have $\eta_{3}-\eta_{2}>\pi/2$ 
leads to
\begin{equation}\label{3ZL:8}
\left\{
	  \begin{array}{ll}
2-\tau_{1}=2\cos^{2}(\eta_{3}-\eta_{2})+\cos(\eta_{3}-\eta_{2})\sqrt{D^{2}-4\sin^{2}(\eta_{3}-\eta_{2})}\\
2-\tau_{2}=2\cos^{2}(\eta_{1}-\eta_{3})+\cos(\eta_{1}-\eta_{3})\sqrt{D^{2}-4\sin^{2}(\eta_{1}-\eta_{3})}\\
2-\tau_{3}=2\cos^{2}(\eta_{2}-\eta_{1})+\cos(\eta_{2}-\eta_{1})\sqrt{D^{2}-4\sin^{2}(\eta_{2}-\eta_{1})}.
\end{array}
\right.
\end{equation}

By replacing $h_{i}$ by its value \eqref{3ZL:4} in \eqref{3ZL:1}, we obtain after calculation
\begin{equation}\label{3ZL:9}
\left\{
	  \begin{array}{ll}
(2-\tau_{3})\cos(\eta_{3}-\eta_{2})&=(2-\tau_{1})\cos(\eta_{1}-\eta_{2})\\
(2-\tau_{2})\cos(\eta_{2}-\eta_{1})&=(2-\tau_{3})\cos(\eta_{3}-\eta_{1})\\
(2-\tau_{1})\cos(\eta_{1}-\eta_{3})&=(2-\tau_{2})\cos(\eta_{2}-\eta_{3}).
\end{array}
\right.
\end{equation}

Finally, replacing $2-\tau_{i}$ by his expression in \eqref{3ZL:9} and using that $\cos(\eta_{i+1}-\eta_{i})\neq0$, we get
\begin{equation}\label{3ZL:10}
\left\{
	  \begin{array}{ll}
2\cos(\eta_{2}-\eta_{1})+\sqrt{D^{2}-4\sin^{2}(\eta_{2}-\eta_{1})}&=2\cos(\eta_{3}-\eta_{2})+\sqrt{D^{2}-4\sin^{2}(\eta_{3}-\eta_{2})}\\
2\cos(\eta_{2}-\eta_{1})+\sqrt{D^{2}-4\sin^{2}(\eta_{2}-\eta_{1})}&=2\cos(\eta_{1}-\eta_{3})+\sqrt{D^{2}-4\sin^{2}(\eta_{1}-\eta_{3})}.
\end{array}
\right.
\end{equation}

Let $f:x\mapsto 2\cos x+\sqrt{D^2-4\sin^{2}x}$.
%\begin{figure}[H]
%\begin{center}
%\includegraphics[width=10cm]{plotf.pdf}
%\caption{Plot of the function $f$ for $D^{2}=6$}\label{fig:plotf}
%\end{center}
%\end{figure}
One easily shows that $f$ is decreasing on $(\pi/2,\pi)$ and hence injective (see Fig.~\ref{fig:plotg}). We thus infer that
$$
\eta_{3}-\eta_{2}=\eta_{1}-\eta_{3}=\eta_{2}-\eta_{1}=\frac{2\pi}{3}.
$$
The triangle $T$ is therefore equilateral and one has $\tau_{1}=\tau_{2}=\tau_{3}=(3+\sqrt{D^{2}-3})/2$.
We recover the smoothed nonagon introduced in Def.~\ref{defKE}.

\medskip

Assume now that $\eta_{3}-\eta_{2}\leq \pi/2$. If $\eta_{3}-\eta_{2}= \pi/2$, then $\tau_{1}=2$ and $M_{1}$ is on in the incircle, which is impossible for $D>2$, otherwise the arc of circle would cross the incircle.

Now we have
\begin{equation}\label{3ZL:10a}
\left\{
	  \begin{array}{ll}
2-\tau_{1}=2\cos^{2}(\eta_{3}-\eta_{2})-\cos(\eta_{3}-\eta_{2})\sqrt{D^{2}-4\sin^{2}(\eta_{3}-\eta_{2})}\\
2-\tau_{2}=2\cos^{2}(\eta_{1}-\eta_{3})+\cos(\eta_{1}-\eta_{3})\sqrt{D^{2}-4\sin^{2}(\eta_{1}-\eta_{3})}\\
2-\tau_{3}=2\cos^{2}(\eta_{2}-\eta_{1})+\cos(\eta_{2}-\eta_{1})\sqrt{D^{2}-4\sin^{2}(\eta_{2}-\eta_{1})}.
\end{array}
\right.
\end{equation}

The same computations as above yield
\begin{equation}\label{3ZL:11}
2\cos(\eta_{3}-\eta_{2})-\sqrt{D^{2}-4\sin^{2}(\eta_{3}-\eta_{2})}=2\cos(\eta_{1}-\eta_{3})+\sqrt{D^{2}-4\sin^{2}(\eta_{1}-\eta_{3})}.
\end{equation}

Now, let us introduce $g:x\mapsto 2\cos x-\sqrt{D^2-4\sin^{2}x}$. 
One easily sees that $g$ is negative while $f$ is positive and therefore, 
the equation $f(x)=g(y)$ has no solution. We conclude that this case cannot happen.

\begin{figure}[H]
\begin{center}
\includegraphics[width=8cm]{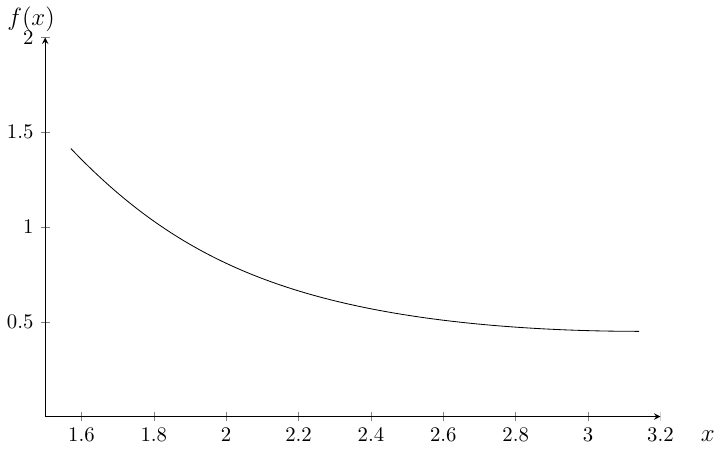}
\includegraphics[width=8cm]{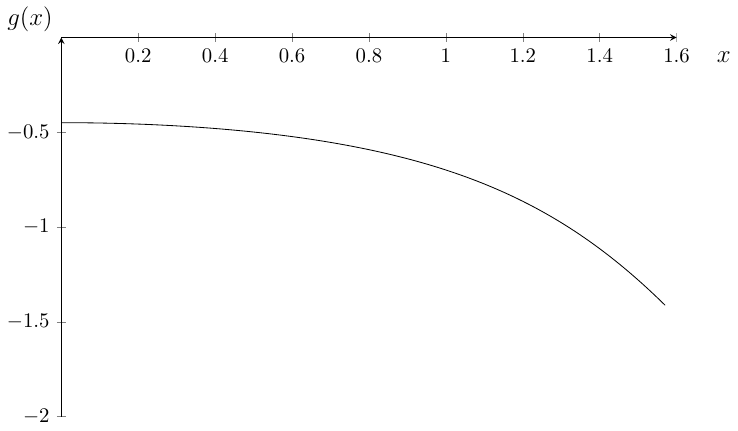}
\caption{$D^{2}=6$. Left: plot of the function $f$. Right: plot of the function $g$.\label{fig:plotg}}
\end{center}
\end{figure}

Finally, the solution for this sub-case is $K_{E}(D)$ defined in Def.~\ref{defKE}. Observe that since $K_E(D)$ is inscribed in the equilateral triangle, we need to have $h< \sqrt{3}$, ie $\tau< 3$ and $D< 2\sqrt{3}$, whence the requirement on $D$ for the sake of the definition of $K_{E}(D)$.
\begin{center}
{\bf Subcase 2: at least one point $M_{i}$ is on the boundary of the free 
zone, namely it is one of the points $A_{j}$ or $B_{j}$.} 
\end{center}
Assume here that a point $M_{i}$, say $M_{1}$ is not in the interior of the free zone. Then $M_{1}=B_{2}$ or $M_{1}=A_{3}$, say $M_{1}=B_{2}$. The free zone $Z_{1}$ is an arc of circle of radius $D/2$ whose antipodal arc is $ \arc{B_{1}M_{3}}$. If $M_{3}$ is also on the boundary of $Z_{3}$ then $Z_{1}$ and $Z_{3}$ would be antipodal and $Z_{2}$ would not have any antipodal arc of circle. This is impossible. So $M_{3}$ lies in the interior of $Z_{3}$ and it has a second arc of circle: $\arc{M_{3}A_{2}}$ which antipodal arc is $\arc{M_{2}B_{3}}$. We claim that $M_{2}=A_{1}$ otherwise $\arc{M_{2}A_{1}}$ would not have antipodal arc.

\begin{figure}[H]
\begin{center}
\includegraphics[width=10cm]{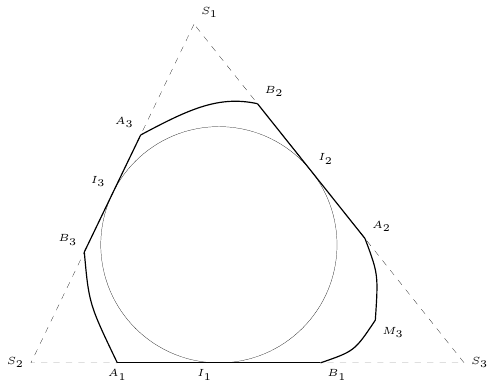}
\caption{An approximate illustration of the case of three free zones and $M_{1}$ in the boundary of the free zone. \label{fig:triangle3ZLMcote}}
\end{center}
\end{figure}

Now, in comparison with the first case, only two vector relation are valid, namely
\begin{equation}\label{Mcote:1}
\overrightarrow{B_1M_3}=\overrightarrow{A_3B_2}\quad \text{and}\quad \overrightarrow{A_2M_3}=\overrightarrow{B_3A_1}.
\end{equation}

Taking the same notations as in the first case with $\tau=\tau_{3}$, one has
\begin{equation}\label{Mcote:2}
\left\{
	  \begin{array}{ll}
\cos\eta_2-h_2\sin\eta_2-\cos\eta_3-h_3\sin\eta_3=(1-\tau)\cos\eta_3-\cos\eta_1+h_1\sin\eta_1\\
\sin\eta_2+h_2\cos\eta_2-\sin\eta_3+h_3\cos\eta_3=(1-\tau)\sin\eta_3-\sin\eta_1-h_1\cos\eta_1
\end{array}
\right.
\end{equation}
and
\begin{equation}\label{Mcote:3}
\left\{
	  \begin{array}{ll}
\cos\eta_1+h_1\sin\eta_1-\cos\eta_3+h_3\sin\eta_3=(1-\tau)\cos\eta_3-\cos\eta_2-h_2\sin\eta_2\\
\sin\eta_2-h_1\cos\eta_1-\sin\eta_3-h_3\cos\eta_3=(1-\tau)\sin\eta_3-\sin\eta_2+h_2\cos\eta_2
\end{array}
\right.
\end{equation}

The same kind of computations as in the first case lead to the following statements:
\begin{equation}\label{Mcote:4}
\left\{
\begin{array}{ll}
\eta_{3}-\eta_{2}=\eta_1-\eta_3=y\\
h_1=h_2\\
2-\tau=2\cos y<0\\
\tau^{2}+h_3^{2}=D^2\\
2h_1=-h_3\cos y.
\end{array}
\right.
\end{equation}

Now set $\eta_{3}=-\pi/2$. then $\eta_{1}=y-\pi/2\in[0,\pi/2]$ and $\eta_{2}=\pi-\eta_{1}$. Observe that$ A_2M_3A_1B_3$ is a rectangle which 
leads to the new equation $\overrightarrow{A_1M_3}\cdot \overrightarrow{A_2M_3}=0$. It rewrites
\begin{equation}\label{Mcote:5}
(\tau-1)^{2}-2(\tau-1)\sin\eta_1+(h_1^{2}+1)(2\sin^{2}\eta_1-1)=0
\end{equation}
and using that
\begin{equation}\label{Mcote:6}
\sin\eta_1=\tau/2-1\quad \text{and}\quad
h_1^{2}+1=\frac{D^{2}-\tau^{2}}{(\tau-2)^{2}}+1,
\end{equation}
Equation \eqref{Mcote:5} becomes
\begin{equation}\label{Mcote:7}
-\tau^{3}+\left(D^{2}/2+5\right)\tau^{2}-\left(2D^{2}+4\right)\tau+D^{2}=0.
\end{equation}
Since $\tau$ has to be a root of the polynomial in $[2,3]$, a calculus argument shows that for $D\in[2,2\sqrt{3}]$, the polynomial has a unique root in $[2,3]$, with $\tau(2)=2$, $\tau(2\sqrt{3})=3$ and $\tau$ is an 
increasing function.

Finally this leads to the construction of the set $K_C(D)$ shown in Fig.~\ref{fig:McoteOpt}.

Furthermore, if we set $t_{1}=\arcsin\left(\frac{2(\sin\eta_{1}+h_1\cos\eta_{1})-\tau+2}{D}\right)$ and $t_2=\arcsin(\tau/D)$ then we have the 
formula
\begin{equation}\label{Mcote:8}
\vert K_C(D)\vert=\frac{\tau}{\tau-2}\sqrt{D^{2}-\tau^{2}}+\frac{D^{2}}{2}(t_2-t_1).
\end{equation}

Let us remark that, using \eqref{Mcote:6}, we have $\cos^2 t_2=(D^2-\tau^2)/D^2$
and $(\tau -2)^2=4\sin^2\eta_1$, thus
$$h_1^2=\frac{D^2-\tau^2}{(\tau-2)^2}=\frac{D^2\cos^2t_2}{4\sin^2\eta_1}\Longrightarrow
h_1=\frac{D\cos t_2}{2\sin\eta_1},$$
and replacing in the definition of $t_1$, it provides the alternative formula
\begin{equation}\label{Mcote:9}
t_1=\arcsin\left(\frac{\cos t_2}{\tan \eta_1}\right).
\end{equation}
\subsection{Comparison}

Now we have to determine what is the optimal shape for a given $D$. Previous analysis show that for $D\geq 2\sqrt{3}$ it is not possible to construct the sets $K_{E}$ and $K_{C}$. Hence the stadium $K_{S}$ is optimal for such $D$.
Let us have a look to the graphics of the area of the three domain for $D\in [2,2\sqrt{3}]$. Now let us investigate the case $D\in [2,2\sqrt{3}]$. 
Graphics \ref{fig:Comparaison1} suggest that, {the inradius $r$ being prescribed}, the set $K_{E}$ is optimal for small values of $D$ and $K_{S}$ is optimal for {large} values of $D$. In the following we prove two facts: 

\begin{enumerate}
\item
The domain $K_C(D)$ is never optimal,
\item
the existence of  $D^{\star}$ such that for $D\leq D^{\star}$, $\vert K_{E}(D)\vert \geq \vert K_{S}(D)\vert$ and for $D\geq D^{\star}$, $\vert K_{S}(D)\vert \geq \vert K_{E}(D)\vert$.
\end{enumerate}

\begin{figure}[H]
\begin{center}
\subfigure[$D\in(2,2.5)$]{\includegraphics[width=8.3cm]{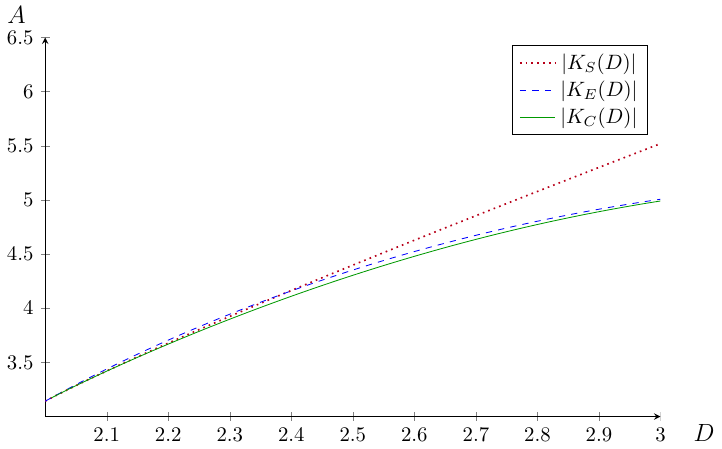}}
\subfigure[$D\in(2,2.1)$]{\includegraphics[width=8.3cm]{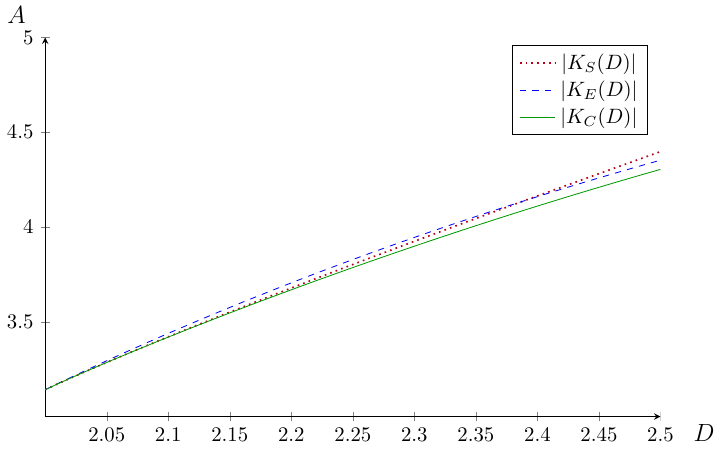}}
\caption {
Comparison of the three areas  \label{fig:Comparaison1}}
\end{center}
\end{figure}

\subsubsection{Proof that $K_C(D)$ is never optimal}
We are going to prove that $K_C(D)<K_S(D)$ for $D\in (2,2\sqrt{3}]$ by comparing their
derivatives (we know that $K_C(2)=K_S(2)=\pi$).
Let us write Equation \eqref{Mcote:7} in the following way
\begin{equation}\label{comp1}
D^2=\tau \frac{\tau^2-5\tau +4}{\frac{\tau^2}{2}-2\tau +1}:=\tau g(\tau)
\end{equation}
where the function $g:x\mapsto (x^2-5x+4)/(x^2/2-2x+1)$ is increasing.
Thus we make the change of variable $D\to \tau$ and rewrite the areas $K_C(D)$ and $K_S(D)$
in terms of $\tau \in [2,3]$. More precisely, we write $\tau=2+h$ with $h\in [0,1]$ and we
write all quantities in term of $h$. Let us observe that
\begin{equation}\label{comp2}
g(2+h)=2+\frac{h}{1-h^2/2}.
\end{equation}

We start with $K_C(D)$ given by \eqref{Mcote:8}. By \eqref{comp2}
$$D^2-\tau^2=(2+h)g(2+h)-(2+h)^2=h^2 \frac{h+h^2/2}{1-h^2/2}$$
and then the first term of $K_C(D)$ is
\begin{equation}\label{comp3}
\frac{\tau}{\tau-2}\sqrt{D^{2}-\tau^{2}} = (2+h) \sqrt{\frac{h+h^2/2}{1-h^2/2}}.
\end{equation}
A simple computation gives its derivative with respect to $h$:
\begin{equation}\label{comp3b}
\frac{d}{d h} \left(\frac{\tau}{\tau-2}\sqrt{D^{2}-\tau^{2}} \right)= \sqrt{\frac{2+h}{2h(1-h^2/2)}}
\frac{1+2h+h^2/2-h^3/2}{1-h^2/2}.
\end{equation}
Now we look at the other term in $K_C(D)$:
\begin{equation}\label{comp4}
t_2=\arcsin\left(\sqrt{\frac{\tau}{g(\tau)}}\right)=\arcsin\left(\sqrt{\frac{2+h-h^2-h^3/2}{2+h-h^2}}\right).
\end{equation}
Now, let us express $t_1$ using \eqref{Mcote:9}:
$$\cos t_2=\sqrt{1-\frac{\tau^2}{D^2}}=\sqrt{1-\frac{2+h}{g(2+h)}}=h\sqrt{\frac{h/2}{2+h-h^2}}$$
while , from $\sin\eta_1=h/2$, we get
$$\tan\eta_1=\sqrt{\frac{1}{1-\sin^2\eta_1}-1}=\frac{h}{4-h^2}.$$
From this, we infer:
\begin{equation}\label{comp5}
t_1=\arcsin\left(\sqrt{\frac{h(2+h)}{2(1+h)}}\right).
\end{equation}
Now using the formula $\arcsin b-\arcsin a= \arcsin\left( b\sqrt{1-a^2}-a\sqrt{1-b^2}\right)$
(all numbers $a$ and $b$ are between $0$ and $1$), we finally get thanks to \eqref{comp4}
and \eqref{comp5}:
\begin{equation}\label{comp6}
t_2-t_1=\arcsin\left((1-h)\sqrt{\frac{2+h}{2-h}}\right).
\end{equation}
In particular, we have
$$\frac{d}{d h}(t_2-t_1)=\frac{h^2-2h-2}{(2-h)\sqrt{2h(2+h)(1-h^2/2)}}.$$
Thus, one has
$$\frac{D^2}{2} \frac{d}{d h}(t_2-t_1)=\frac{(1+h)(2-h)(2+h)}{2(1-h^2/2)} \frac{h^2-2h-2}{(2-h)\sqrt{2h(2+h)(1-h^2/2)}}=$$
$$=\sqrt{\frac{2+h}{2h(1-h^2/2)}}
\frac{-1-2h-h^2/2+h^3/2}{1-h^2/2}$$
which is exactly the opposite of \eqref{comp3b}.
Therefore
$$\frac{d}{d h} K_C(2+h)=D \frac{d D}{dh} \arcsin\left((1-h)\sqrt{\frac{2+h}{2-h}}\right).$$

On the other hand, since 
$$\frac{d}{dD} K_S(D)= D\arcsin\left(\frac{2}{D}\right)$$ we have
$$\frac{d}{d h} K_S(2+h)=D \frac{d D}{dh} \arcsin\left(\frac{2}{D}\right)$$
and to compare the derivatives, it suffices to compare the arguments in the $\arcsin$.
Now
$$\frac{2}{D}=\sqrt{\frac{2(1-h^2/2)}{(1+h)(2-h)(2+h)}}$$
and squaring and simplifying amounts to prove
$$\frac{4(1-h^2/2)}{(1+h)(2+h)}\,>\, (1-h)^2(2+h) \Leftrightarrow h^2(5+h-3h^2-h^3)>0$$
which is true for $0<h\leq  1$. This finishes the proof of $K_S(D)>K_C(D)$ for $D>2$.

\subsubsection{Existence of  $D^{\star}$}

Note that $\vert K_S(2)\vert =\vert K_E(2)\vert =\pi$. Now we compute 
the derivative of $D\mapsto \vert K_E(D)\vert -\vert K_S(D)\vert $ which is given by
\[
\frac{d}{dD}(\vert K_E(D)\vert -\vert K_S(D)\vert)=\frac{3}{2}\times D\left(\frac{2\pi}{3}-2\arccos(\frac{\sqrt{3}}{D})\right)-D\arcsin\left(\frac{2}{D}\right),
\]
which has the same sign as $\pi-3\arccos(\frac{\sqrt{3}}{D}))-\arcsin(2/D)=g(D)$.

Now, we have 
\[
g'(D)=-\frac{3\sqrt{3}}{D\sqrt{D^{2}-3}}+\frac{2}{D\sqrt{D^2-4}}
\]
which is positive if and only if $D\in[2,\sqrt{\frac{96}{23}}]$. Together 
with $g(0)=0$ and $g(2\sqrt{3})=-\arcsin(\frac{1}{\sqrt{3}})<0$ we get that $g$ is positive then negative. Finally we deduce that  $D\mapsto \vert K_E(D)\vert -\vert K_S(D)\vert $ is increasing then decreasing with value 0 at $2$ and taking negative value at $2\sqrt{3}$. We finally get the existence of some $D^{\star}\in[2,2\sqrt{3}]$ such that for $D\leq D^{\star}$,  $\vert K_E(D)\vert \geq \vert K_S(D)\vert $ and for $D\geq D^{\star}$,  $\vert K_E(D)\vert \leq \vert K_S(D)\vert $. This conclude the proof.

\appendix
\section{Proof of Lemma \ref{lemopcond}}\label{sec:prooflemopcond}
To prove the Lemma, we will introduce an auxiliary problem whose unknown is the restriction of $R$ to the set $I$. Let us introduce $J=[0,2\pi]\backslash (I\cup (I+\pi))$. Let us decompose $R$ as $R=R_0\mathds{1}_J+u^\star\mathds{1}_I+(D-u^\star(\cdot-\pi))\mathds{1}_{I+\pi}$, and observe that
$$
J(R)=\int_I (2F[R]u^\star-DF[R]-Du^\star+D^2)+\int_J F[R]R_0.
$$
and 
$$
\int_Iu^\star(\theta)\cos \theta\, d\theta=\alpha \quad\text{and}\quad
\int_Iu^\star(\theta)\sin \theta\, d\theta=\beta,
$$
with $\alpha=-\frac12\int_J R_0(\theta)\cos \theta\, d\theta+\frac{D}{2}\int_I\cos \theta\, d\theta$ and $\beta=-\frac12\int_J R_0(\theta)\sin 
\theta\, d\theta+\frac{D}{2}\int_I\sin \theta\, d\theta$.

We will now characterize $u^\star$ by exploiting that it solves the optimization problem
\begin{equation}\label{pbaux:metz1600}
\sup_{u\in \widetilde{\mathcal{R}}_D}\widetilde{J}(u)\quad \text{where}\quad \widetilde{J}(u)=\int_I (2h u-Dh-Du+D^2)+\int_J hR_0,
\end{equation}
where $h$ solves the ODE 
\begin{equation}\label{hODE1705}
\left\{\begin{array}{ll}
h+h''= R_0\mathds{1}_J+u\mathds{1}_I+(D-u(\cdot-\pi))\mathds{1}_{I+\pi} 
& \text{in }(0,2\pi)\\
\int_0^{2\pi}h(\theta)e^{i\theta}\, d\theta=0 & \\
h(0)=h(2\pi), \ h'(0)=h'(2\pi) & 
\end{array}
\right.
\end{equation}
and
$$
\widetilde{\mathcal{R}}_D=\left\{u\in L^\infty(I;[0,D])\mid \int_I u(\theta)e^{i\theta}\, d\theta=\alpha+i\beta \right\}.
$$
Let us now derive the first order necessary optimality conditions for this problem. Since the method is standard, we briefly comment on the method 
allowing us to write such conditions: first, the mapping $\widetilde{\mathcal{R}}_D\ni u\mapsto h$, where $h$ solves \eqref{hODE1705}, being linear it is G\^ateaux-differentiable at $u^\star$ in every direction $\xi$ belonging to the tangent cone to the set $\widetilde{\mathcal{R}}_D$ at $u^\star$. Furthermore, its differential $\dot h$ is the unique solution of the ODE
$$
\left\{\begin{array}{ll}
\dot h+\dot h''= \xi \mathds{1}_I-\xi(\cdot-\pi)\mathds{1}_{I+\pi} & \text{in }(0,2\pi)\\
\int_0^{2\pi}\dot h (\theta)e^{i\theta}\, d\theta=0 & \\
\dot h(0)=\dot h(2\pi), \ \dot h'(0)=\dot h'(2\pi). & 
\end{array}
\right.
$$ 
It follows that $\widetilde{\mathcal{R}}_D\ni u\mapsto \widetilde{J}(u)$ is G\^ateaux-differentiable at $u^\star$ and its differential reads
\begin{eqnarray*}
\langle d\widetilde{J}(u^\star),\xi\rangle &=& \lim_{\eta\searrow 0}\frac{\widetilde{J}(u^\star+\eta \xi)-\widetilde{J}(u^\star)}{\eta}=\int_I 
(2\dot hu^\star+2h\xi -D\dot h-D\xi)+\int_J\dot hR_0\\
&=& \int_I(2h-D)\xi + \int_I \dot h (2u^\star-D) +\int_J \dot h R_0=2\int_I(2h-D)\xi ,
\end{eqnarray*}
by using several times integration by parts and the relation $h(\theta)+h(\theta+\pi)=D$ on $I$.

We now have to deal with two kinds of constraints in $\widetilde{\mathcal{R}}_D$: a global $L^1$ one and point-wise ones, since $u$ belongs to $[0,D]$ almost everywhere. Although such constraints are standard, we briefly explain how to derive the Euler inequation for this problem with the help of a penalization approach, for the sake of completeness. For $\varepsilon>0$, let us introduce $\widetilde{J}_\varepsilon$ as the penalized functional 
$$
\widetilde{J}_\varepsilon (u)= \widetilde{J}(u)+\frac{1}{\varepsilon}\left|\int_I u(\theta)e^{i\theta}\, d\theta-(\alpha+i\beta)\right|^2.
$$
We consider the optimization problem
\begin{equation}\label{pbpen1805}
\sup_{u\in L^\infty(I;[0,D]) }\widetilde{J}_\varepsilon (u).
\end{equation}
On what follows, we will need to consider an element $\xi$ to the tangent 
cone $\mathcal{T}_{u}$ to $L^\infty(I;[0,D])$ at $u$, that we describe hereafter.
%\begin{definition} (\cite[chapter 7]{HenrPierre})\label{def:tgtcone}
%For every $u\in L^\infty(I;[0,D])$, the tangent cone to the set $L^\infty(I;[0,D])$ at $u$, also called the \textit{admissible cone} to the set $L^\infty(I;[0,D])$ at $u$, denoted $\mathcal{T}_{u}$ is the set of functions $\xi\in L^\infty(I)$ such that, for any sequence of positive real numbers $(\eta_n)_{n\in \N}$ decreasing to $0$, there exists a sequence of functions $\xi_n\in L^\infty(I)$ converging to $\xi$ as $n\rightarrow +\infty$, and $u+\eta_n\xi_n\in L^\infty(I;[0,D])$ for every $n\in\N$.\label{footnote:cone}
%\end{definition}

Since they follow from a basic variational analysis, we do not provide all the details to the following claims:
\begin{itemize}
\item Since $L^\infty(I;[0,D])$ is compact for the weak-star convergence in $L^\infty$, the resolvent operator $\widetilde{\mathcal{R}}_D\ni u\mapsto h\in L^2(\mathbb{T})$ is compact and therefore,
the penalized problem \eqref{pbpen1805} has a solution $u_\varepsilon \in 
L^\infty(I;[0,D])$. 
\item Let $h_\varepsilon$ be the solution to \eqref{hODE1705} associated to $u_\varepsilon$. There exists a sequence $(\varepsilon_n)_{n\in \N}$ decreasing to 0, there exists $\widetilde{u}\in L^\infty(I;[0,D])$ such that $(u_{\varepsilon_n})_{n\in \N}$ converges weakly-star to $\widetilde u$ in $L^\infty(I;[0,D])$ and $(h_{\varepsilon_n})_{n\in \N}$ converges strongly to $\widetilde h\in H^1(0,2\pi)$ and uniformly in $\mathscr{C}^0([0,2\pi])$ as $n\to +\infty$. Furthermore, one has necessarily $\int_I u_{\varepsilon_n}(\theta)e^{i\theta}\, d\theta=\alpha+i\beta)+\operatorname{O}(\varepsilon_n)$ and therefore, $\widetilde{u}$ belongs to $\widetilde{\mathcal{R}}_D$.
\item Let $\xi\in \mathcal{T}_{\widetilde u}$. There exists $\xi_n\in \mathcal{T}_{u_{\varepsilon_n}}$ such that $(\xi_n)_{n\in \N}$ converges weakly-star to $\xi$ as $n\to +\infty$ (this follows from the definition of the tangent cone and the fact that pointwise inequalities are preserved by the weak-star convergence).  
\end{itemize}

Let $\xi \in \mathcal{T}_{\widetilde{u}}$. According to the computations above, the necessary first order optimality conditions for the penalized problem \eqref{pbpen1805} read: for every $n\in \N$, since $\xi \in \mathcal{T}_{u_{\varepsilon_n}}$, one has 
$$
\int_I\left(2h_{\varepsilon_n}(\theta)-D+\alpha_n\cos \theta+\beta_n\sin \theta\right)\xi_n (\theta)\, d\theta\leq 0,
$$
where
$$
\alpha_n=\frac{1}{\varepsilon_n}\left(\int_I u_{\varepsilon_n}(s)\cos{s}\, ds -\alpha\right)\quad \text{and}\quad
\beta_n=\frac{1}{\varepsilon}\left(\int_I u_{\varepsilon_n}(s)\sin{s}\, 
ds -\beta\right).
$$
Let us divide the inequality above by $\sqrt{1+\alpha_n^2+\beta_n^2}$. Since the quantities $\sqrt{1+\alpha_n^2+\beta_n^2}$,  \\
$\alpha_n/\sqrt{1+\alpha_n^2+\beta_n^2}$ and $\beta_n/\sqrt{1+\alpha_n^2+\beta_n^2}$ are uniformly bounded with respect to $n$, one can assume that they respectively converge (up to a new extraction) to $\mu\geq 0$, $\bar \alpha\in \R$ and $\bar \beta\in \R$ such that $(\mu, \bar \alpha,\bar\beta)\neq (0,0,0)$. Since $\xi$ was arbitrarily chosen, by passing to the limit as $n\to +\infty$, we get at the end that the first order necessary conditions associated to Problem~\eqref{hODE1705} read
\begin{equation}x
\forall \xi\in \mathcal{T}_{\widetilde{u}}, \quad \int_I\left(\mu(2\widetilde{h}(\theta)-D)+\bar \alpha \cos \theta+\bar \beta \sin \theta\right)\xi (\theta)\, d\theta\leq 0.
\end{equation}
Now, since $\widetilde{J}_{\varepsilon_n}(u)= \widetilde{J}(u)$ for every $u\in \widetilde{\mathcal{R}}_D$, it follows that
$$
\widetilde{J}_{\varepsilon_n}(u_\varepsilon)=\max_{u\in L^\infty(I;[0,D])}\widetilde{J}_{\varepsilon_n}(u)\geq \max_{u\in L^\infty(I;[0,D])}\widetilde{J}(u)=\widetilde{J}(u^\star)\geq \widetilde{J}(u_{\varepsilon_n}).
$$
Passing to the limit in this inequality yields $\widetilde{J}(u^\star)\geq \widetilde{J}(\widetilde{u})$. Using that $\widetilde{u}$ belongs to $\widetilde{\mathcal{R}}_D$, we infer that $\widetilde{u}$ solves Problem~\eqref{pbaux:metz1600}. Therefore, we can assume without loss of generality that $\widetilde{u}=u^\star$.

\section*{Acknowledgements} 
We want to thank warmly the anonymous referee who allows us to improve the writing and clarity of this paper.
All three authors were partially supported by the ANR Project ANR-18-CE40-0013 SHAPO ``SHAPe Optimization''. The third author was partially supported by the Project ``Analysis and simulation of optimal shapes - application to lifesciences'' of the Paris City Hall.

\bibliographystyle{abbrv}
\bibliography{mybibfile}

\end{document}